\documentclass[11pt,letterpaper]{amsart}
\usepackage{hyperref}
\usepackage{amsfonts,mathrsfs,bbm,latexsym,rawfonts,amsmath,amssymb,amsthm,epsfig}
\usepackage{fullpage, setspace, color}

\newtheorem{thm}{Theorem}[section]

\newtheorem{rem}[thm]{Remark}
\newtheorem{lem}[thm]{Lemma}
\newtheorem{prop}[thm]{Proposition}

\numberwithin{equation}{section}

\newcommand{\al}{\alpha}

\newcommand{\ld}{\lambda}

\newcommand{\de}{\delta}
\newcommand{\De}{\Delta}
\newcommand{\ep}{\varepsilon}

\newcommand{\om}{\omega}
\newcommand{\Om}{\Omega}
\newcommand{\ga}{\gamma}
\newcommand{\Ga}{\Gamma}

\renewcommand{\P}{\mathcal{P}}

\renewcommand{\S}{\mathscr{S}}
\newcommand{\g}{\mathfrak{g}}

\newcommand{\U}{\mathbb{S}}

\DeclareMathOperator{\tr}{tr}

\newcommand{\Real}{\mathbb{R}}

\newcommand{\norm}[1]{\Vert#1\Vert}

\def\<{\left\langle} \def\>{\right\rangle}
\def\({\left(} \def\){\right)}
\newcommand{\n}{\nabla}
\newcommand{\p}{\partial}

\renewcommand{\S}{\mathscr{S}}

\begin{document}
\title{Existence and uniqueness of local regular solution to the Schr\"{o}dinger flow from a bounded domain in $\mathbb{R}^3$ into $\U^2$}
%\thanks{*Corresponding Author}
\author{Bo Chen}
\address{Department of Mathematical Sciences, Tsinghua University, Beijing, 100084, China}
\email{chenbo@mail.tsinghua.edu.cn}
	
\author{Youde Wang}
\address{1. School of Mathematics and Information Sciences, Guangzhou University;
		2. Hua Loo-Keng Key Laboratory of Mathematics, Institute of Mathematics, AMSS, and School of
		Mathematical Sciences, UCAS, Beijing 100190, China.}
\email{wyd@math.ac.cn}
%\thanks{*Corresponding Author}
%\date{\today}
\begin{abstract}
In this paper, we show the existence and uniqueness of local regular solutions to the initial-Neumann boundary value problem of Schr\"{o}dinger flow from a smooth bounded domain $\Om\subset \Real^3$ into $\U^2$(namely Landau-Lifshitz equation without dissipation). The proof is built on a parabolic perturbation method, an intrinsic geometric energy argument and some observations on the behaviors of some geometric quantities on the boundary of the domain manifold. It is based on methods from Ding and Wang (one of the authors of this paper) for the Schr\"odinger flows of maps from a closed Riemannian manifold into a K\"ahler manifold as well as on methods by Carbou and Jizzini for solutions of the Landau-Lifshitz equation.
\end{abstract}
\maketitle
%\tableofcontents
\section{Introduction}\label{intr}
In this paper, we are concerned with the existence and uniqueness of strong solutions to the initial-Neumann boundary value problem of the following Schr\"odinger flow from a smooth bounded domain $\Om\subset\mathbb{R}^3$ into $\U^2$
$$\p_tu=u\times\Delta u,$$
where ``$\times$" denotes the cross product in $\Real^{3}$. It is well-known that this equation is also called the Landau-Lifshitz equation with a long history.

\subsection{Definitions and Background}
In physics, the Landau-Lifshitz (LL) equation is a fundamental evolution equation for the ferromagnetic spin chain and was proposed on the phenomenological ground in studying the dispersive theory of magnetization of ferromagnets. It was first deduced by Landau and Lifshitz in \cite{LL}, and then proposed by Gilbert in \cite{G} with dissipation. It is well-known that the Landau and Lifshitz system is of fundamental importance in theory of magnetization and ferromagnets, and has extensive applications in physics.

In fact, this equation describes the Hamiltonian dynamics corresponding to the Landau-Lifshitz energy, which is defined as follows. Consider a ferromagnetic body occupying a bounded, possibly multi-connected domain $\Omega$ of the Euclidean space $\mathbb{R}^3$. We neglect mechanical effects due to magnetization (magnetostriction) and assume the temperature to be constant and lower than Curie's temperatures. Suppose that $\Om\subset\mathbb{R}^3$ is a bounded smooth domain. Let $u$, denoting magnetization vector, be a mapping from $\Omega$ into a unit sphere $\U^2\subset\mathbb{R}^3$. The Landau-Lifshitz energy functional of a map $u:\Om\to\U^2$ is defined by
$$\mathcal{E}(u):=\int_{\Om}\Phi(u)\,dx+\frac{1}{2}\int_{\Om}|\n u|^2\,dx-\frac{1}{2}\int_{\Omega}h_d\cdot u\,dx.$$
Here $\n$ denote the gradient operator and $dx$ is the volume element of $\mathbb{R}^3$.
	
For the above Landau-Lifshitz functional, the first and second terms are the anisotropy and exchange energies, respectively, and $\Phi(u)$ is a smooth function on $\U^2$. The last term is the self-induced energy, and $h_d(u)$ is the demagnetizing field, which has the following form
$$h_d(u)(x) =-\n \int_{\Omega} \nabla N(x-y)u(y)dy,$$
where $N(x) = -\frac{1}{4\pi|x|}$ is the Newtonian potential in $\mathbb{R}^3$.
	
The Landau-Lifshitz (for short LL) equation without dissipation can be written as
\begin{equation}\label{LL-eq}
\p_t u=-u\times h
\end{equation}
where the local field $h$ of $\mathcal{E}(u)$ can be derived as
\[
h:=-\frac{\delta\mathcal{E}(u)}{\delta u}= \Delta u + h_d -\nabla_u\Phi.
\]
	
In this paper we want to consider the existence of regular solution to equation \eqref{LL-eq}. Since the anisotropy term $\Phi$, non-local field $h_d(u)$  and the negative sign``$-$'' in equation \eqref{LL-eq} do not affect on our analysis and main conclusions, for the sake of convenience and simplicity, we only consider the classical Schr\"{o}dinger flow into $\U^2$ (Landau-Lifshitz)
\[\p_tu =u\times\De u.\]

Intrinsicly, ``$u\times$" can be considered as a complex structure $J$ on $\U^2$, who rotates anticlockwise the tangent space at $u$ by an angle of $\frac{\pi}{2}$ degrees. Therefore, we can write the above equation as
\[\p_tu = J(u)P(u)(\De u),\]
where $P(u): \mathbb{R}^3\rightarrow T_u\U^2$ is a standard projection operator.

From the viewpoint of infinite dimensional symplectic geometry, Ding-Wang \cite{DW} proposed to consider the so-called Schr\"odinger flows for maps from a Riemannian manifold into a symplectic manifold, which can be regarded as an extension of LL equation \eqref{LL-eq} and was also independently introduced from the viewpoint of integrable systems by Terng and Uhlenbeck in \cite{Uh1}. Namely, suppose $(M,g)$ is a Riemannian manifold, $(N,J,\om)$ is a symplectic manifold, the Schr\"odinger flow is a time-dependent map $u:M\times \Real^+\to N\hookrightarrow \Real^{n+k}$ satisfying
\[\p_t u=J(u)\tau(u)\]
where $$\tau(u)=\De_g u+A(u)(\n u,\n u)$$ is the tension field of $u$, where $A(u)(\cdot, \cdot)$ is the second fundamental form of $N$ in $\Real^{n+k}$. Here we always embed isometrically $(N,J,\om)$ in an Euclidean space $\Real^{n+k}$ where $n=\dim(N)$.

\medskip
A great deal of effort has been devoted to the study of Landau-Lifshitz equation defined on an Euclidean spaces or a flat torus (closed manifold) in the last five decades. One has made great progress in the PDE aspects of the Schr\"odinger flows containing the existence, uniqueness and regularities of various kinds of solutions. Now, we recall some known results which are closely related to our work in the present paper.

In 1986, P.L. Sulem, C. Sulem and C. Bardos in \cite{SSB} employed difference method to prove that the Schr\"{o}dinger flow for maps from $\Real^n$ into $\U^2$ admits a global weak solution or a smooth local solution under suitable initial value conditions. Moreover, they also addressed the existence of global smooth solution if the initial value map is small. In 1998, Y.D. Wang( the second named author) \cite{W} obtained the global existence of weak solution to the Schr\"odinger flow for maps from a closed Riemannian manifold or a bounded domain in $\Real^n$ into $\U^2$ by adopting a more geometric approximation equation than the Ginzburg-Landau penalized equation used for the LLG equation in \cite{AS,B,T}. Later, Z.L. Jia and Y.D. Wang \cite{JW1,JW2} employed a method originated from \cite{W}  to achieve the global weak solutions to a large class of generalized Schr\"odinger flows in more general setting, where the base manifold is a bounded domain $\Om\subset\Real^n$($n\geq 2$) or a compact Riemannian manifold $M^n$ and the target space is $\U^2$ or the unit sphere $\U^n_\g$ in a compact Lie algebra $\g$. However, the existence of global weak solution for the Schr\"odinger flows between manifolds are still open.

The local existence of the Schr\"odinger flow from a general closed Riemannian manifold into a K\"ahler manifold was first obtained by Ding and the second named author of this paper in \cite{DW}. By using a parabolic approximation and the intrinsic geometric energy method, they proved that, if M is an $m$ dimensional compact Riemannian manifold or the Euclidean space $\mathbb{R}^m$ and the initial map $u_0 \in W^{k,2}(M, N)$ with $k \geq[m/2] + 2$, then there exists a local solution $u \in L^\infty([0, T), W^{k,2}(M, N))$. The local regular(smooth) solution to the Schr\"odinger flow from $\Real^n$ into a K\"ahler manifold was also addressed by Ding and Wang in \cite{DW1} (Later, Kenig, Lamm, Pollack, Staffilani and Toro in \cite{KLPST} also provided another different approach). Furthermore, they also obtained the persistence of regularity results, in that the solution always stays as regular as the initial data (as measured in Sobolev norms), provided that one is within the time of existence guaranteed by the local existence theorem. In proving its well-posedness, the heart of the matter is resolved by estimating multi-linear forms of some intrinsic geometric quantities.

For low-regularity initial data, the initial value problem for Schr\"odinger flow from an Euclidean space into $\mathbb{S}^2$ has been studied indirectly using the ``modified Schr\"odinger map equations" and certain enhanced energy methods, for instance, A.R. Nahmod, A. Stefanov and K. K. Uhlenbeck \cite{NSU} have ever used the standard technique of Picard iteration in some suitable function spaces of the Schr\"odinger equation to obtain a near-optimal (but conditional) local well-posedness result for the Schr\"odinger map flow equation from two dimensions into a Riemann surface $X$, in the model cases of the standard sphere $X = \mathbb{S}^2$ or hyperbolic space $X = \mathbb{H}^2$. In proving its well-posedness, the heart of the matter is resolved by considering truly quatrilinear forms of weighted $L^2$-functions.

For one dimensional global existence for Schr\"odinger flow from $\mathbb{S}^1$ or $\mathbb{R}^1$ into a K\"ahler manifold, we refer to \cite{PWW, RRS} and references therein. The global well-posedness result for the Schr\"odinger flow from $\mathbb{R}^n$ ($n\geq 3$) into $\U^2$ in critical Besov spaces was proved by Ionescu and Kenig in \cite{IK}, and independently by Bejanaru in \cite{B1}, and then was improved to global regularity for small data in the critical Sobolev spaces in dimensions $n\geq4$ in \cite{BIK}. Finally, in \cite{BIKT} the global well-posedness result for the Schr\"odinger flow for small data in the critical Sobolev spaces in dimensions $n\geq2$ was addressed. Recently, Z. Li in \cite{L1, L2} proved that the Schr\"odinger flow from $\mathbb{R}^n$ with $n\geq 2$ to compact K\"ahler manifold with small initial data in critical Sobolev spaces is global, which solves the open problem raised in \cite{BIKT}.

\medskip
On the contrary, F. Merle, P. Rapha\"el and I. Rodnianski \cite{MRR} also considered the energy critical Schr\"odinger flow problem with the 2-sphere target for equivariant initial data of homotopy index $k=1$. They showed the existence of a codimension one set of smooth well localized initial data arbitrarily close to the ground state harmonic map in the energy critical norm, which generates finite time blowup solutions, and gave a sharp description of the corresponding singularity formation which occurs by concentration of a universal bubble of energy. One also found some self-similar solutions to Schr\"odinger flow from $\mathbb{C}^n$ into $\mathbb{C}P^n$ with local bounded energy which blow up at finite time, for more details we refer to \cite{DTZ, GSZ, NSVZ}.

\medskip
As for some travelling wave solutions with vortex structures, F. Lin and J. Wei \cite{LW} employed perturbation method to consider such solutions for the Schr\"odinger map flow equation with easy-axis and proved the existence of smooth travelling waves with bounded energy if the velocity of travelling wave is small enough. Moreover, they showed the travelling wave solution has exactly two vortices. Later, J. Wei and J. Yang \cite{WY} considered the same Schr\"odinger map flow equation as in \cite{LW}, i.e. the Landau-Lifshitz equation describing the planar ferromagnets. They constructed a travelling wave solution possessing vortex helix structures for this equation. Using the perturbation approach, they give a complete characterization of the asymptotic behaviour of the solution.

\medskip
On the other hand, since the seventies of the 20th century magnetic domains have been the object of a considerable research from the applicative viewpoint (e.g., see \cite{S, V}), especially because of invention of ``magnetic bubbles" devices and their use in computer hardware. In the literature, physicists and mathematicians are always interested in the Landau-Lifshitz-Gilbert system with Neumann boundary conditions(see \cite{CF, S-R}), for instance, Carbou and Jizzini considered a model of ferromagnetic material subject to an electric current, and proved the local in time existence of very regular solutions for this model in
the scale of $H^k$ spaces. In particular, they described in detail the compatibility conditions at the boundary for the initial data, for details we refer to \cite{CJ}. Roughly speaking, Carbou and Jizzini showed that
\begin{equation*}
\begin{cases}
\p_tu =-u\times(u\times\De u) + \alpha u\times\De u,\quad\quad&\text{(x,t)}\in\Om\times \Real^+,\\[1ex]
\frac{\p u}{\p \nu}=0, &\text{(x,t)}\in\p\Om\times \Real^+,\\[1ex]
u(x,0)=u_0: \Om\to \U^2,
\end{cases}
\end{equation*}
admits a very regular solution if $u_0$ meets some compatibility conditions at the boundary, where $\alpha$ is a real number.

\medskip
A natural problem is if the following
\begin{equation*}
\begin{cases}
\p_tu =u\times\De u,\quad\quad&\text{(x,t)}\in\Om\times \Real^+,\\[1ex]
\frac{\p u}{\p \nu}=0, &\text{(x,t)}\in\p\Om\times \Real^+,\\[1ex]
u(x,0)=u_0: \Om\to \U^2,
\end{cases}
\end{equation*}
where $\nu$ is the outer normal vector on $\Om$ and $u_{0}$ is the initial value map, admits a strong or regular solution? The results on the existence of global weak solutions proved by Wang in \cite{W} hints us that the initial-Neumann boundary value problem of the Schr\"odinger flow should be posed as follows
\begin{equation}\label{S-eq}
\begin{cases}
\p_tu =u\times\De u,\quad\quad&\text{(x,t)}\in\Om\times \Real^+,\\[1ex]
\frac{\p u}{\p \nu}=0, &\text{(x,t)}\in\p\Om\times \Real^+,\\[1ex]
u(x,0)=u_0: \Om\to \U^2,\quad \frac{\p u_0}{\p \nu}|_{\p \Om}=0.
\end{cases}
\end{equation}

Our goal of this paper is to prove the above problem (\ref{S-eq}) admits a local in time strong (regular) solution. The above Schr\"odinger flow with starting manifold is a bounded domain $\Om\subset \Real^{n}$ with $n\geq 2$ is a challenging problem, there is few results on the well-posedness of initial-Neumann boundary value problem \eqref{S-eq} in the literature.

\subsection{Strategy and Main Results}
In the present paper, we are intend to studying the local well-posedness of the above equation \eqref{S-eq}. However, the method involving harmonic analysis in $\Real^n$ used in \cite{B,B1,BIK,IK} may not be effective for the equation in a bounded domain or a complete compact manifold. And hence, we still apply a similar parabolic perturbation approximation and intrinsic geometric energy method with that in \cite{DW} since Carbou and Jizzini \cite{CJ} have shown that the corresponding problem of the approximating equation is well-posed. Because of the space of test functions associated to equation \eqref{S-eq} is much smaller than that for the setting in \cite{DW}, it is more difficult for us to get the desired geometric energy estimates on the solutions of \eqref{S-eq}. We need to overcome some new essential difficulties caused by the boundary of domain manifold.

Our main conclusions can be presented as follows
	
\begin{thm}\label{thm1}
Let $\Omega$ be a smooth bounded domain in $\mathbb{R}^3$. Assume that the initial value maps $u_0\in H^{3}(\Omega, \U^2)$ with $\frac{\p u_0}{\p \nu}|_{\p \Om}=0$. Then, there exists a positive number $T_0>0$ depending only on $\norm{u_0}_{H^3}$ and the geometry of $\Omega$ such that the equation \eqref{S-eq} admits a unique regular solution $u\in L^\infty([0,T_0], H^3(\Om,\U^2))$ with initial value $u_0$.
\end{thm}

\begin{rem}\
	
\begin{enumerate}	
\item Theorem \ref{thm1} still holds for the case of $\Om$ being a smooth bounded domain in $\mathbb{R}^2$ (also see \cite{PWW1, PWW2}).

\item If the initial map $u_0$ satisfies some furthermore compatibility and applicable regularity conditions, we also prove \eqref{S-eq} admits a unique smooth solution $u\in L^\infty([0,T_0], H^3(\Om,\U^2))$ with initial value $u_0$. We will present these results in another paper \cite{C-W1}.

\item In forthcoming papers we will extend the results in Theorem \ref{thm1} to the case the starting manifold of the Schr\"odinger flow is a $2$ or $3$ dimensional compact Riemannian maniflod with smooth boundary and the target manifold is a compact K\"ahler manifold.

\item It is still open for the case the dimension of the starting manifold is greater than three.
\end{enumerate}
\end{rem}

On the other hand, we recall that  the self-induced vector field is defined by
\[h_d(u):=-\n\int_\Om\n N(x-y)u(u)\,dx\]
in the sense of distributions. Hence, the following estimates of $h_d$ is a fundamental result in theory of singular integral operators, its proof can be found in \cite{CF,CJ, C-W}.

\begin{prop}\label{es-h_d}
	Let $p\in(1,\infty)$ and $\Om$ be a smooth bounded  domain in $\Real^{3}$. Assume that $u \in W^{k,p}(\Om,\Real^{3})$ for $k\in \mathbb{N}$. Then, the restriction of $h_d(u)$ to $\Om$ belongs to $W^{k,p}(\Om, \Real^{3})$. Moreover, there exists constants $C_{k,p}$, which is independent of $u$, such that
	$$\lVert h_d(u)\rVert_{W^{k,p}(\Om)}\leq C_{k,p}\lVert u\rVert_{W^{k,p}(\Om)}.$$
	In fact, $h_{d}:W^{k,p}(\Om, \Real^{3})\to W^{k,p}(\Om, \Real^{3})$ is a  bounded linear  operator.
\end{prop}

With Proposition \ref{es-h_d} at hand, we take an almost same argument as that in the proof of Theorem \ref{thm1} to conclude the following result.
 \begin{thm}\label{thm2}
 Let $\Omega$ be a smooth bounded  domain in $\mathbb{R}^3$. Assume that the initial value maps $u_0\in H^{3}(\Omega, \U^2)$ with $\frac{\p u_0}{\p \nu}|_{\Om}=0$, and $\Phi\in C^\infty(\U^2)$. Then, there exist a positive number $T_0>0$ depending only on $\norm{u_0}_{H^3}$ and the geometry of $\Omega$ such that the initial-Neuman boundary value problem of equation \eqref{LL-eq} admits a unique regular solution $u\in L^\infty([0,T_0], H^3(\Om,\U^2))$ with initial value $u_0$.
 \end{thm}

As mentioned in above, the proof of Theorem \ref{thm1} 	is divided into three steps. In the first step, we consider the parabolic perturbation approximation to \eqref{S-eq}
\begin{equation}\label{PS-eq0}
\begin{cases}
\p_tu_\ep =\ep(\De u_\ep+|\n u_\ep|^2 u_\ep)+u_\ep\times\De u_\ep,\quad\quad&\text{(x,t)}\in\Om\times \Real^+,\\[1ex]
\frac{\p u_\ep}{\p \nu}=0, &\text{(x,t)}\in\p\Om\times \Real^+,\\[1ex]
u_\ep(x,0)=u_0: \Om\to \U^2,\quad \frac{\p u_0}{\p \nu}|_{\p \Om}=0.
\end{cases}
\end{equation}
Here $\ep\in(0, 1)$ is the perturbation constant. Recall that the local well-posedness for this parabolic system is established recently in \cite{CJ}, which can be formulated as the following proposition (also see our recent work \cite{C-W}).

\begin{prop}\label{str-sol}
Suppose that $u_0\in H^3(\Om, \U^2)$, there exists a positive number $T_\ep$ depending only on $\ep$ and $ \norm{u_0}_{H^2}$ such that the equation \eqref{PS-eq0} admits a unique regular solution $u_\ep$ on $\Om\times [0,T_\ep)$ which satisfies for any $T<T_\ep$ that
\begin{itemize}
	\item[$(1)$] $|u_\ep(x,t)|=1$ for all $(x,t)\in [0,T]\times\Om$;
	\item[$(2)$] $u_\ep\in L^\infty([0,T],H^3(\Om))\cap L^{2}([0,T],H^4(\Om))$.
\end{itemize}
\end{prop}

In Section \ref{s1}, we will extend $T_\ep$ stated in this proposition to the maximal existence time of solution $u_\ep$ satisfying the above properties.
\medskip

In the second step, we get the uniform $H^3$-energy estimates of $u_\ep$ (``uniform'' means  ``independent of $\ep\in(0, 1)$''). Our basic idea is to make full use of the integrality of complex structure to get higher geometric energy estimates of solution $u_\ep$ to
\[\p_t u=\ep \tau(u)+J(u)\tau(u)\]
as in \cite{DW}. Inspired by this idea we consider the corresponding equation of $\p_t u_\ep$ from the viewpoint of intrinsic geometry. For simplicity we let $M$ be a bounded domain $\Om\subset \Real^n$ and choose the natural coordinates $x=\{x_1,\cdots,x_n\}$ on $\Om$. Noting that $\n J=0$, we take a simple calculation to show
\begin{equation}\label{wave1}
\begin{aligned}
&\n_{\p t}\p_t u_\ep+(1-\ep^2)\n_i\n_i(\tau(u_\ep))-2\ep J(u_\ep)\n_i\n_i(\tau(u_\ep))\\
=&\ep^2R^N(\tau(u_\ep), \n_i u_\ep)\n_i u_\ep+\ep J(u_\ep)R^N(\tau(u_\ep),\n_i u_\ep)\n_i u_\ep\\
&+\ep R^N(J(u_\ep)\tau(u_\ep),\n_i u_\ep)\n_i u_\ep+J(u_\ep)R^N(J(u_\ep)\tau(u_\ep),\n_i u_\ep)\n_i u_\ep.
\end{aligned}
\end{equation}
Here, $\n_i=\n^N_{\frac{\p u}{\p x_i}}$ and $R^N$ is the Riemannian curvature tensor of $N$.

In particular, if the target manifold $N=\U^2$ with the complex struture $J(u)=u\times$, then it is not difficult to show that $u_\ep$ satisfies the following extrinsic formula (to see Formula \eqref{w-eq2} for precise form)
\begin{equation}\label{wave2}
\begin{aligned}
&\frac{\p^2 u_\ep}{\p t^2}+(1-\ep^2)\De \tau(u_\ep)-2\ep\De(u_\ep\times \De u_\ep)\\
=&\ep\{\frac{\p}{\p t}(|\n u_\ep|^2u_\ep)-2\textnormal{\mbox{div}}(\n u_\ep\dot{\times} \n^2 u_\ep)+u_\ep\times \De(|\n u_\ep|^2u_\ep)+|\n u_\ep|^2u_{\ep}\times \De u_\ep\}\\
&+\De (|\n u_\ep|^2 u_\ep)-2\textnormal{\mbox{div}}^2((\n u_\ep\dot{\otimes}\n u_\ep)u_\ep)+2\textnormal{\mbox{div}}((\n u_\ep\dot{\otimes}\n u_\ep)\cdot\n u_\ep)-\textnormal{\mbox{div}}(|\n u_\ep|^2\n u_\ep),
\end{aligned}
\end{equation}
by using the fact $|u|=1$ and
\[R^{\U^2}(X, Y)Z=\<Y,Z\>X-\<X,Z\>Y\]
for any $X,Y$ and $Z$ in $\Ga(T\U^2)$.

Let $\ep=0$. After contracting some terms in the above extrinsic formula, we obtain again the nice equation in \cite{SSB}(also to see Formula \eqref{w-inq5})
\begin{equation}\label{wave3}
\begin{aligned}
\frac{\p^2u }{\p t^2}+\De^2u
=&-2\textnormal{\mbox{div}}^2((\n u\dot{\otimes}\n u)u)+2\textnormal{\mbox{div}}((\n u\dot{\otimes}\n u)\cdot\n u)\\
&-\textnormal{\mbox{div}}(|\n u|^2\n u).
\end{aligned}
\end{equation}

Then, by taking $\p_t u_\ep$ and $\p_t \tau(u_\ep)$ as test functions of equation \eqref{wave2} respectively, a delicate computation shows the desired $H^3$-estimates of $u_\ep$ holds true on a uniform time interval $[0,T_0]$ with $0<T_0<T_\ep$. Therefore, by letting $\ep\to 0$, the existence part of Theorem \ref{thm1} is proved.
\medskip

In the last step, we show the uniqueness of the solution $u$ we obtained  by adopting the intrinsic energy method introduced in \cite{Mc,S-W}.

\subsection{A Related Problem}

Recently, Chern et al \cite{Chern} described a new approach for the purely Eulerian simulation of incompressible fluids. In their setting, the fluid state is represented by a $\mathbb{C}^2$-valued wave function evolving under the Schr\"odinger equation subject to incompressibility constraints. The underlying dynamical system is Hamiltonian and governed by the kinetic energy of the fluid together with an energy of Landau-Lifshitz type. They deduced the following
$$\partial_t u + \mathcal{L}_vu=\tilde{\alpha}(u\times\Delta u),$$
where $\tilde{\alpha}$ is a real number, $u:\Omega\times[0, T)\rightarrow\mathbb{S}^2$ and $\mathcal{L}_v$ is the Lie derivative with respect to the field $v$ on $\Omega$ with $\mbox{div}(v)\equiv0$. They called this dynamical system as incompressible Schr\"odinger flow.

If $\Omega\subset\mathbb{R}^3$ is a smooth bounded domain, $\mathcal{L}_v$ is just the $\nabla_v$, and hence the above can be written as
$$\partial_t u + \nabla_vu=\tilde{\alpha}(u\times\Delta u),$$
with $\mbox{div}(v)\equiv 0$ on $\Omega$.

For simplicity, we let $\tilde{\alpha}=1$ and consider the following initial-Neumann boundary value problem of incompressible Schr\"odinger flow
\begin{equation}\label{IS-eq}
\begin{cases}
\p_tu + \nabla_vu=u\times\De u,\quad\quad&\text{(x,t)}\in\Om\times \Real^+,\\[1ex]
\frac{\p u}{\p \nu}=0, &\text{(x,t)}\in\p\Om\times \Real^+,\\[1ex]
u(x,0)=u_0: \Om\to \U^2,\quad \frac{\p u_0}{\p \nu}|_{\p \Om}=0.
\end{cases}
\end{equation}
Here we always suppose that the field $v$ is smooth enough and $\mbox{div}(v)\equiv 0$ on $\Omega$.

In order to prove the local in time well-posedness, the first thing that needs to be done is to establish the local existence to the following
\begin{equation}\label{IS-eq*}
\begin{cases}
\p_tu + \nabla_vu=\epsilon u\times(u\times\De u) + u\times\De u,\quad\quad&\text{(x,t)}\in\Om\times \Real^+,\\[1ex]
\frac{\p u}{\p \nu}=0, &\text{(x,t)}\in\p\Om\times \Real^+,\\[1ex]
u(x,0)=u_0: \Om\to \U^2,\quad \frac{\p u_0}{\p \nu}|_{\p \Om}=0.
\end{cases}
\end{equation}
Then, by letting $\ep\rightarrow 0$ we can also obtain some similar results with that for (\ref{S-eq}) and the proof goes almost the same as that stated in Theorem \ref{thm1} except for that we need to treat the term $\nabla_v u$. Since $\mbox{div}(v)\equiv 0$ on $\Omega$, and if we additionally provide that $v$ satisfies the boundary compatibility condition
\[\<v,\nu\>|_{\p\Om\times \Real^+}\equiv 0,\]
the term $\nabla_v u$ does not cause any essential difficulties.
Here we do not elucidate the details.

It is worthy to point out that Huang \cite{H} has ever concerned the coupled system of Navier-Stokes equation and incompressible Schr\"odinger flow defined on a closed manifold or $\mathbb{R}^n$, and shown the local in time existence of the initial value problem of this system in some suitable Sobolev spaces.

\medskip
The rest of our paper is organized as follows. In Section \ref{s: pre}, we introduce the basic notations on Sobolev space and some critical preliminary lemmas. In Section \ref{s1} and Section \ref{s2}, we give the proof of local existence of regular solution to \eqref{S-eq} stated in Theorem \ref{thm1}. The uniqueness will built up in Section \ref{s3}. We close with two appendices. First, this is the locally regular estimates of the approximate solution $u_\ep$. Secondly, it is the characterisation and formulation of the Schr\"odinger flow in moving frame and parallel transport.

\section{Preliminary}\label{s: pre}
In this section, we first recall some notations on Sobolev spaces, which will be used in whole context. Let $u=(u_1,u_2,u_3):\Om\to\U^{2}\hookrightarrow\Real^3$ be a map. We set
$$H^{k}(\Om,\U^{2})=\{u\in W^{k,2}(\Om,\Real^{3}):|u|=1\,\,\text{for a.e. x}\in \Om\}.$$
Moreover, let $(B,\norm{.}_B)$ be a Banach space and $f:[0,T]\to B$ be a map. For any $p>0$ and $T>0$, we define
\[\norm{f}_{L^p([0,T], B)}:=\(\int_{0}^{T}\norm{f}^p_{B}dt\)^{\frac{1}{p}},\]
and
\[L^p([0,T],B):=\{f:[0,T]\to B:\norm{f}_{L^p([0,T],B)}<\infty\}.\]
In particular, we denote
\[L^{p}([0,T],H^{k}(\Om,\U^{2}))=\{u\in L^{p}([0,T],H^{k}(\Om,\Real^{3})):|u|=1\,\,\text{for a.e. (x,t)}\in \Om\times[0,T]\},\]
where $k,\,l\in \mathbb{N}$  and $p\geq 1$.

Next, we need to recall some crucial preliminary lemmas which we will use later. The $L^2$ theory of Laplace operator with Neumann boundary condition implies the following Lemma of equivalent norms, to see \cite{Weh}.
\begin{lem}\label{eq-norm}
Let $\Om$ be a bounded smooth domain in $\Real^{m}$ and $k\in \mathbb{N}$. There exists a constant $C_{k,m}$ such that, for all $u\in H^{k+2}(\Om)$ with $\frac{\p u}{\p \nu}|_{\p\Om}=0$,
\begin{equation}\label{eq-n}
\norm{u}_{H^{2+k}(\Om)}\leq C_{k,m}(\norm{u}_{L^{2}(\Om)}+\norm{\De u}_{H^{k}(\Om)}).
\end{equation}	
\end{lem}

In particular, the above lemma implies that we can define the $H^{k+2}$-norm of $u$ as follows
$$\norm{u}_{H^{k+2}(\Om)}:=\norm{u}_{L^2(\Om)}+\norm{\De u}_{H^k(\Om)}.$$
	
In order to show the uniform estimates and the convergence of solutions to the approximate equation constructed in the coming sections, we also need to use the Gronwall inequality and the classical compactness results in \cite{BF, Sim}.
	
\begin{lem}\label{Gron-inq}
Let $f: \Real^+\to \Real^+$ be a nondecreasing continuous function such that $f>0$ on $(0,\infty)$ and $\int_{1}^{\infty}\frac{1}{f}dx<\infty$. Let $y$ be a continuous function which is nonnegative on $\Real^+$ and let $g$ be a nonnegative function in $L^{1}_{loc}(\Real^+)$. We assume that there exists a positive number $y_0>0$ such that for all $t\geq0$, we have the inequality
\[y(t)\leq y_0+\int_{0}^{t}g(s)ds+\int_{0}^{t}f(y(s))ds.\]
Then, there exists a positive number $T^*$ depending only on $y_0$ and $f$, such that for all $T<T^*$, there holds true
\[\sup_{0\leq t\leq T}y(t)\leq C(T,y_0)\]
for some constant $C(T,y_0)$.
\end{lem}
	
\begin{lem}[Aubin-Lions-Simon compact Lemma, see Theorem II.5.16 in \cite{BF} or \cite{Sim}]\label{A-S}
Let $X\subset B\subset Y$ be Banach spaces with compact embedding $X\hookrightarrow B$. Let $1\leq p, \,q, \,r\leq \infty$. For $T>0$, we define
\[E_{p,r}=\{f\in L^{p}((0,T), X), \frac{d f}{dt}\in L^{r}((0,T), Y)\},\]
then the following properties holds
\begin{itemize}
\item[$(1)$] If $p< \infty$ and $p<q$, the embedding $E_{p,r}\cap L^q((0,T), B)$ in $L^s((0,T), B)$ is compact for all $1\leq s<q$.
\item[$(2)$] If $p=\infty$ and $r>1$, the embedding of $E_{p,r}$ in $C^0([0,T], B)$ is compact.
\end{itemize}
\end{lem}
\begin{lem}[Theorem II.5.14 in \cite{BF}]\label{C^0-em}
Let $k\in \mathbb{N}$, then the space
\[E_{2,2}=\{f\in L^{2}((0,T),H^{k+2}(\Om)),\frac{\p f}{\p t}\in L^{2}((0,T), H^k(\Om))\}\]
is continuously embedded in $C^0([0,T], H^{k+1}(\Om))$.
\end{lem}

To end this section, we briefly introduce the notations of Galerkin basis and Galerkin projection. Let $\Om$ be a  smooth bounded  domain in $\Real^m$, $\ld_{i}$ be the $i^{th}$ eigenvalue of the operator $\De-I$  with Neumann boundary condition, whose corresponding eigenfunction is $f_{i}$. That is,
$$(\De-I)f_{i}=-\ld_{i}f_{i}\,\,\,\,\quad\text{with}\quad\,\,\,\,\frac{\p f_{i}}{\p \nu}|_{\p\Om}=0.$$
	
Without loss of generality, we assume $\{f_{i}\}_{i=1}^{\infty}$ are completely standard orthonormal basis of $L^{2}(\Om,\Real^{1})$. Let $H_{n}=span\{f_{1},\dots f_{n}\}$ be a finite subspace of $L^{2}$, $P_{n}:L^{2}(\Om, \Real^1)\to H_{n}$ be the canonical projection. In fact, for any $f\in L^{2}$, we define
\[f^{n}=P_{n}f=\sum_{1}^{n}\<f,f_{i}\>_{L^{2}}f_{i},\]
then,
\[\lim_{n\to \infty}\norm{f-f_{n}}_{L^{2}}=0.\]
	
	%For this canonical projection, we have the following uniform estimates, which is essential for our method to get very regularity of solution in section \ref{s2}. Its proof can be found in \cite{CJ}.
	%\begin{lem}\label{es-P_n}
	%	There exists a constant $C$ such that for all $n$, the projection $P_n$ satisfies the following properties:
	%	\begin{enumerate}
	%		\item For all $f\in H^4(\Om,\Real^1)$ such that $\frac{\p f}{\p \nu}|_{\p \Om}=0$ and $\frac{\p \De f}{\p \nu}|_{\p\Om}=0$, there holds
	%		 \[\norm{P_n(f)}_{H^4(\Om)}\leq C\norm{f}_{H^5(\Om)},\]
	%		\item Moreover, if $f\in H^5(\Om,\Real^1)$, we have
	%		\[\norm{P_n(f)}_{H^5(\Om)}\leq C\norm{f}_{H^5(\Om)}.\]
	%	\end{enumerate}
	%\end{lem}
	
\section{Parabolic perturbation to the Schr\"{o}dinger flow}\label{s1}
In this section, we consider the parabolic perturbation to the Schr\"{o}dinger flow \eqref{S-eq}:
\begin{equation}\label{PS-eq}
\begin{cases}
\p_tu_\ep =\ep(\De u_\ep+|\n u_\ep|^2 u_\ep)+u_\ep\times\De u_\ep,\quad\quad&\text{(x,t)}\in\Om\times \Real^+,\\[1ex]
\frac{\p u_\ep}{\p \nu}=0, &\text{(x,t)}\in\p\Om\times \Real^+,\\[1ex]
u_\ep(x,0)=u_0: \Om\to \U^2,\quad \frac{\p u_0}{\p \nu}|_{\p \Om}=0.
\end{cases}
\end{equation}
Here $\ep\in (0,1)$ is the perturbation constant. The local in time regular solution to equation \eqref{PS-eq} is established in \cite{CJ,CF,C-W} by virtue of Galerkin approximation method in the following theorem, the proof of which can also be found in \cite{C-W}.
\begin{thm}\label{uin-ex-PS}
Suppose that $u_0\in H^3(\Om)$. Then there exists a maximal existence time $T_\ep$ depending on $\norm{u_0}_{H^2}$ such that equation \eqref{PS-eq} admits a unique regular solution $u_\ep$ on $\Om\times [0,T_\ep)$ which satisfies that for any $T<T_\ep$
\begin{itemize}
\item[$(1)$] $|u_\ep(x,t)|=1$ for all $(x,t)\in [0,T]\times\Om$;
\item[$(2)$] $u_\ep\in L^\infty([0,T],H^3(\Om))\cap L^{2}([0,T],H^4(\Om))$;
\item[$(3)$] $\frac{\p u_\ep}{\p t}\in L^\infty([0,T],H^1(\Om))\cap L^{2}([0,T],H^2(\Om))$ and $\frac{\p^2 u_\ep}{\p t^2}\in L^{2}([0,T],L^2(\Om))$.
\end{itemize}
Moreover, there exists a constant $C(T)>0$ such that
\begin{align}\label{es-PS}
\sup_{t\leq T}(\norm{u_\ep}^2_{H^3(\Om)}+\norm{\frac{\p u_\ep}{\p t}}^2_{H^1(\Om)})+\int_{0}^{T}(\norm{u_\ep}^2_{H^4(\Om)}+\norm{\frac{\p u_\ep}{\p t}}^2_{H^2(\Om)}+\norm{\frac{\p^2 u_\ep}{\p t^2}}^2_{L^2(\Om)})dt\leq C(T).
\end{align}
\end{thm}
\begin{proof}
Let $u^n_\ep$ be the solution to the Galerkin approximation equation to \eqref{PS-eq}:
\begin{equation}\label{G-PS-eq1}
\begin{cases}
\p_tu^n_\ep =\ep\De u_\ep+\ep P_n(|\n u^n_\ep|^2 u^n_\ep)+P_n(u^n_\ep\times\De u^n_\ep),\quad\quad&\text{(x,t)}\in\Om\times \Real^+,\\[1ex]
u^n_\ep(x,0)=\sum_{i=1}^n\int_{\Om}\<u_0,f_i\>dxf_i(x),\quad &\text{x}\in\Om.
\end{cases}
\end{equation}
		
By establishing some delicate energy estimates and taking a process of convergence for $u^n_\ep$ as $n\to \infty$, we can obtain a solution $u_1$ to equation \eqref{PS-eq} on $\Om\times[0,T_1)$ for some $T_1>0$ depending only on $\ep$ and $\norm{u_0}_{H^2(\Om)}$, which satisfies estimate (\ref{es-PS}) for all $T<T_1$ (cf. \cite{CJ,C-W}).
		
Without loss of generality, we assume that $T_1$ is not the maximal existence time, then the solution $u_1$ satisfies
\[\sup_{t< T_1}(\norm{u_1}^2_{H^3(\Om)}+\norm{\frac{\p u_1}{\p t}}^2_{H^1(\Om)})+\int_{0}^{T_1}(\norm{u_1}^2_{H^4(\Om)}+\norm{\frac{\p u_1}{\p t}}^2_{H^2(\Om)}+\norm{\frac{\p^2 u_1}{\p t^2}}^2_{L^2(\Om)})dt<\infty.\]
It implies
\begin{align}\label{C^0-es1}
u_1\in C^0([0,T_1], H^3(\Om)),
\end{align}
\begin{align}\label{C^0-es2}
\frac{\p u_1}{\p t}\in C^0([0,T_1], H^1(\Om)),
\end{align}
by applying Lemma \ref{C^0-em}. Then, there exist maps $u_{T_1}(x)\in H^3(\Om)$  and $v_{T_1}(x)\in H^1(\Om)$ such that
\[\lim_{T\to T_1}\norm{u_1(\cdot,t)-u_{T_1}}_{H^3(\Om)}=0,\]
and
\[\lim_{T\to T_1}\norm{\p_tu_1(\cdot,t)-v_{T_1}}_{H^1(\Om)}=0.\]
		
Next, we need to show that $u_{T_1}(x)$ satisfies the Neumann boundary condition, i.e.,
\[\frac{\p u_{T_1}}{\p\nu}|_{\p \Om}=0.\]
Since $\frac{\p u_1(x,t)}{\p\nu}|_{\p \Om\times[0,T_1)}=0$, it follows that
\[\int_{0}^{T_1}\int_{\Om}\<\De u_1,\phi\>dxdt+\int_{0}^{T_1}\int_{\Om}\<\n u_1,\n \phi\>dxdt=0,\]
for all $\phi\in C^\infty(\bar{\Om}\times [0,T_1))$. In particular, if we choose $\phi(x,t)=f(x)\eta(t)$ for all $f\in C^\infty(\bar{\Om})$ and $\eta\in C^\infty[0,T_1)$, and denote $g(t)=\int_{\Om}\<\De u_1,f\>+\<\n u_1,\n f\>dx$, then $g(t)\in C^0[0,T_1]$ and there holds
\[\int_{0}^{T_1}g(t)\eta(t)dt=0,\]
which follows $g(t)\equiv0$, that is
\[\int_{\Om}\<\De u_1,f\>+\<\n u_1,\n f\>dx\equiv0.\]
And hence we get
\[\frac{\p u_{T_1}}{\p\nu}|_{\p \Om}=0\]
as $t\to T_1$. Here we have used the fact (\ref{C^0-es1}).
		
Therefore, by taking the same argument as that in above, we conclude that there exists a regular solution $w_1$ to equation \eqref{PS-eq} on $\Om\times [T_1,T_2)$ by replacing $u_0$ with $u_{T_1}$. Thus, it is not difficult to show the map
\begin{equation*}
u_2(x,t)=
\begin{cases}
u_1(x,t) &(x,t)\in \Om\times[0,T_1),\\[1ex]
w_1(x,t) &(x,t)\in \Om\times[T_1,T_2),
\end{cases}
\end{equation*}
is a solution to equation \eqref{PS-eq} on $\Om\times[0,T_2)$ satisfying Estimates (\ref{es-PS}) on $\Om\times [0, T_2)$, since we have the compactness (\ref{C^0-es2}). By repeating this process, we can get a maximal solution $u_\ep$ on $\Om\times [0,T_\ep)$, which satisfies Estimate (\ref{es-PS}).
		
Finally, as in \cite{CJ,C-W} we can finish the proof of uniqueness by considering the equation of the difference of two solutions and applying a direct energy method.
\end{proof}

\begin{rem}\label{es-u_ep}
\begin{itemize}
\item[$(1)$] Since $T_\ep$ is the maximal existence time, if $T_\ep<\infty$, then we have
\begin{align*}
&\sup_{t<T_\ep}(\norm{u_\ep}^2_{H^3(\Om)}+\norm{\frac{\p u_\ep}{\p t}}^2_{H^1(\Om)})\\
&+\lim_{T\to T_\ep}\{\int_{0}^{T}\int_{\Om}(\norm{u_\ep}^2_{H^4(\Om)}+\norm{\frac{\p u_\ep}{\p t}}^2_{H^2(\Om)}+\norm{\frac{\p^2 u_\ep}{\p t^2}}^2_{L^2(\Om)})dt\}=\infty.
\end{align*}

\item[$(2)$] Let $\{T_i\}$ be the existence times of Galerkin approximation solution constructed in the above Theorem \ref{uin-ex-PS}. If we set
\[S:=\{T_1, T_2,\dots, T_i, \dots\},\]
then for all $[T,T^\prime]\subset [0,T_\ep)\setminus S$, the estimates in \cite{CJ, C-W}also imply the Galerkin approximation solution $u^n_\ep$ satisfies
\[\sup_{T\leq t\leq T^\prime}(\norm{u^n_\ep}^2_{H^2(\Om)}+\norm{\frac{\p u^n_\ep}{\p t}}^2_{H^1(\Om)})+\int_{0}^{T}\int_{\Om}(\norm{u^n_\ep}^2_{H^3(\Om)}+\norm{\frac{\p u^n_\ep}{\p t}}^2_{H^2(\Om)})dt\leq C(T,T^\prime).\]
\end{itemize}
\end{rem}	

Next, we show the uniform estimates of the solution $u_\ep$, which is independent of $\ep$, and hence obtain a regular solution to the Schr\"{o}dinger flow \eqref{S-eq} by taking limit of the sequence of approximation solutions $\{u_\ep\}$ as $\ep\to 0$.
	
First of all, by choosing $u_\ep$ and $-\De u_\ep$ as test functions for equation \eqref{PS-eq}, we can show the uniform $H^1$-estimates as follows.
\begin{equation}\label{H^1}
\frac{\p}{\p t}\int_{\Om}(|u_\ep|^2+|\n u_\ep|^2)dx+2\ep\int_{\Om}|u_\ep\times\De u_\ep|^2dx=0.
\end{equation}

\subsection{Uniform $H^2$-estimates}\label{H^2-es}
However, to show directly the further uniform $H^2$-estimate on $u_\ep$ by usual energy estimates seems difficult, because of the spin term of
\[u_\ep\times\De u_\ep.\]
To proceed, we need to show the wave-type formula to parabolic perturbation of the Schr\"{o}dinger flow in the below lemma, which is mentioned in Section \ref{intr}.
\begin{lem}\label{p-wave-str}
Let $u_\ep$ be the regular solution to equation \eqref{PS-eq} obtained in the above. Then the following properties hold true
\begin{itemize}
\item[$(1)$] For a.e. $(x,t)\in \Om\times[0,T_\ep)$, we have
\begin{equation}\label{w-eq1}
\frac{\p^2 u_\ep}{\p t^2}=\ep(\De \frac{\p u_\ep}{\p t}+\frac{\p}{\p t}(|\n u_\ep|^2u_\ep))+\frac{\p u_\ep}{\p t}\times \De u_\ep+u_\ep \times \De \frac{\p u_\ep}{\p t};
\end{equation}
\item[$(2)$] For a.e. $(x,t)\in\Om\times[0,T_\ep)$, we have
\begin{equation}\label{w-eq2}
\begin{aligned}
&\frac{\p^2 u_\ep}{\p t^2}+(1-\ep^2)\De(\tau(u_\ep))-2\ep\De(u_\ep\times \De u_\ep)\\
=&\ep\{\frac{\p}{\p t}(|\n u_\ep|^2u_\ep)-2\textnormal{\mbox{div}}(\n u_\ep\dot{\times} \n^2 u_\ep)+u_\ep\times \De(|\n u_\ep|^2u_\ep)+|\n u_\ep|^2u_{\ep}\times \De u_\ep\}\\
&+\De(|\n u_\ep|^2 u_\ep)-2\textnormal{\mbox{div}}^2((\n u_\ep\dot{\otimes}\n u_\ep))u_\ep-2\<\n|\n u_\ep|^2, \n u_\ep\>\\
&-2\<\De u_\ep, \n u_\ep\>\cdot \n u_\ep-|\n u_\ep|^2\De u_\ep,
\end{aligned}
\end{equation}
where
\[\textnormal{\mbox{div}}(\n u_\ep\dot{\times} \n^2 u_\ep):=\sum_{i,j=1}^3\p_i(\p_ju_\ep\times\p_{ij}u_\ep),\]
\[\textnormal{\mbox{div}}^2((\n u_\ep\dot{\otimes}\n u_\ep))u_\ep:=\sum_{i,j=1}^3\p_{ij}((\p_iu_\ep\cdot\p_{j}u_\ep))u_\ep.\]
\end{itemize}
\end{lem}
\begin{proof}
The formula in $(1)$ is given directly by equation \eqref{PS-eq}. We need only to show the equation in $(2)$. A simple calculation gives
\begin{equation}\label{w-inq0}
\begin{aligned}
&\frac{\p^2 u_\ep}{\p t^2}-\ep^2\De(\De u_\ep+|\n u_\ep|^2u_\ep)\\
=&\ep\{\frac{\p}{\p t}(|\n u_\ep|^2 u_\ep)+\De(u_\ep \times \De u_\ep)+u_\ep\times\De^2 u_\ep+u_\ep\times\De(|\n u_\ep|^2u_\ep)+|\n u_\ep|^2u_\ep\times \De u_\ep\}\\
&+\{(u_\ep\times \De u_\ep)\times \De u_\ep+u_\ep\times(u_\ep \times \De^2 u_\ep)+2u_\ep\times(\n u_\ep\dot{\times}\n \De u_\ep)\}\\
=&\ep I+II
\end{aligned}
\end{equation}
For the term $u_\ep\times\De^2 u_\ep$ in part $I$, there holds
\[u_\ep\times\De^2 u_\ep=\De (u_\ep\times \De u_\ep)-2\mbox{div}(\n u_\ep \dot{\times}\n^2 u_\ep).\]
		
Next, we turn to presenting the calculation of $II$. By applying the Lagrangian formula
\[a\times(b\times c)=\<a,c\>b-\<a,b\>c\]
and
\[(a\times b)\times c=\<a,c\>b-\<b,c\>a\]
for any vectors $a,b,c$ in $\Real^3$, we have
\begin{equation}\label{w-inq1}
\begin{aligned}
(u_\ep\times \De u_\ep)\times \De u_\ep=&\<u_\ep,\De u_\ep\>\De u_\ep-|\De u_\ep|^2u_\ep\\
=&-|\n u_\ep|^2\De u_\ep-|\De u_\ep|^2u_\ep,
\end{aligned}
\end{equation}
\begin{equation}\label{w-inq2}
\begin{aligned}
u_\ep\times (u_\ep\times \De^2 u_\ep)=&\<\De^2 u_\ep, u_\ep\>u_\ep-\De^2 u_\ep\\
=&-(|\De u_\ep|^2 +2|\n^2 u_\ep|^2+4\<\n \De u_\ep,\n u_\ep\>)u_\ep-\De^2 u_\ep,
\end{aligned}
\end{equation}
and
\begin{equation}\label{w-inq3}
\begin{aligned}
u_\ep\times(\n u_\ep\dot{\times}\n \De u_\ep)=&\sum_{i=1}^3u_\ep\times(\p_i u_\ep\times\p_i \De u_\ep)\\
=&\<\p_i\De u_\ep, u_\ep\>\p_i u_\ep-\<\p_i u_\ep, u_\ep\>\p_i \De u_\ep\\
=&\<\n\De u_\ep, u_\ep\>\n u_\ep\\
=& -\<\n |\n u_\ep|^2,\n u_\ep\>-\<\De u_\ep, \n u_\ep\>\cdot \n u_\ep.
\end{aligned}
\end{equation}
Here, we have used the fact $|u_\ep|\equiv 1$.

Moreover, we have
\begin{equation}\label{w-inq4}
\<\De^2 u_\ep,u_\ep\>-|\De u_\ep|^2=-2(|\De u_\ep|^2+|\n^2 u_\ep|^2+2\<\n \De u_\ep,\n u_\ep\>)=-2\mbox{div}^2(\n u_\ep \dot{\otimes}\n u_\ep).
\end{equation}
By combining the above equations \eqref{w-inq1}-\eqref{w-inq4} with \eqref{w-inq0}, we get the desire formula in $(2)$.
\end{proof}
\begin{rem}
Let $\ep=0$, we have
\[\frac{\p^2 u}{\p t^2}+\De^2 u
=-2\textnormal{\mbox{div}}^2((\n u\dot{\otimes}\n u))u-2\<\n|\n u|^2, \n u\>-2\<\De u, \n u\>\cdot \n u-|\n u|^2\De u.\]
And hence, a tedious but direct calculation again gives the wave-type structure of the Schr\"{o}dinger flow in Section \ref{intr}:
\begin{equation}\label{w-inq5}
\begin{aligned}
\frac{\p^2 u}{\p t^2}+\De^2u	
=&-2\textnormal{\mbox{div}}^2((\n u\dot{\otimes}\n u)u)\\
&+2\textnormal{\mbox{div}}((\n u\dot{\otimes}\n u)\cdot\n u)-\textnormal{\mbox{div}}(|\n u|^2\n u).
\end{aligned}
\end{equation}
Here
\[\textnormal{\mbox{div}}^2((\n u\dot{\otimes}\n u)u):=\sum_{i,j=1}^3\p_{ij}((\p_iu\cdot\p_{j}u)u),\]
and
\[\textnormal{\mbox{div}}((\n u\dot{\otimes}\n u)\cdot\n u):=\sum_{i,j=1}^3\p_i((\p_iu\cdot\p_{j}u)\p_j u).\]
\end{rem}
	
We also need the following compatibility conditions to equation \eqref{w-eq1}(or \eqref{w-eq2})  on the boundary and an equivalent $H^3$-norm of the solution $u_\ep$ to equation \eqref{PS-eq}, which is related to the $H^1$-norm of $\frac{\p u_\ep}{\p t}$.
\begin{lem}\label{comp-bdy}
The solution $u_\ep$ satisfies the following compatibility conditions on the boundary:
\begin{equation}
\begin{cases}
\frac{\p }{\p \nu}\frac{\p u_\ep}{\p t}|_{\p\Om\times [0,T_\ep)}=0,\\[1ex]
\frac{\p \tau(u_\ep)}{\p \nu}|_{\p\Om\times [0,T_\ep)}=0,
\end{cases}
\end{equation}
in the sense of trace.
\end{lem}
\begin{proof}
By using the results in Theorem \ref{uin-ex-PS} and the estimates in Remark \ref{es-u_ep}, without loss of generality, we assume the Galerkin approximation solution $u^n_{\ep}$ satisfies
		
\[\frac{\p u^{n}_\ep}{\p t} \to \frac{\p u_\ep}{\p t}\quad \text{strongly in}\quad L^2([T,T^\prime],H^1(\Om)),\]
and
\[\frac{\p u^{n}_\ep}{\p t} \rightharpoonup \frac{\p u_\ep}{\p t}\quad \text{weakly in}\quad L^2([T,T^\prime],H^2(\Om)),\]
for $[T,T^\prime]\subset[0,T_\ep)\setminus S$.
		
Let $\phi$ be a given function in $C^\infty(\bar{\Om}\times[0,T_\ep))$. Since $$\frac{\p}{\p \nu}\frac{\p u^n_{\ep}}{\p t}|_{\p \Om\times[0,T_\ep)}=0,$$ it follows that there holds
\[\int_T^{T^{\prime}}\int_{\Om}\<\De \frac{\p u^n_\ep}{\p t}, \phi\>dxdt+\int_T^{T^{\prime}}\int_{\Om}\<\n \frac{\p u^n_\ep}{\p t},\n \phi\>dxdt=0.\]
Letting $n\to \infty$, we can see immediately that
\[\int_T^{T^{\prime}}\int_{\Om}\<\De \frac{\p u_\ep}{\p t}, \phi\>dxdt+\int_T^{T^{\prime}}\int_{\Om}\<\n \frac{\p u_\ep}{\p t},\n \phi\>dxdt=0.\]
Since $\frac{\p u_\ep}{\p t} \in L^2_{loc}([0,T_\ep), H^2(\Om))$, it follows for any $0<T<T_\ep$, there holds
\begin{equation}\label{eq}
\int_0^{T}\int_{\Om}\<\De \frac{\p u_\ep}{\p t}, \phi\>dxdt=-\int_0^{T}\int_{\Om}\<\n \frac{\p u_\ep}{\p t},\n \phi\>dxdt,
\end{equation}
that is $$\frac{\p }{\p \nu}\frac{\p u_\ep}{\p t}|_{\p\Om\times [0,T_\ep)}=0.$$
		
In consideration of the fact that $\n \tau(u_\ep)$ is orthogonal to $u_\ep \times \n \tau(u_\ep)$, we can see that the equation
\[\frac{\p u_\ep}{\p t}=\ep\tau(u_\ep)+u_\ep\times \tau(u_\ep)\]
implies $\frac{\p \tau(u_\ep)}{\p \nu}|_{\p\Om\times [0,T_\ep)}=0$ in the sense of trace, since $\frac{\p u_\ep}{\p \nu}|_{\p \Om \times [0,T_\ep)}=0$.
\medskip

Now we provide the details of the proof for $\frac{\p \tau(u_\ep)}{\p \nu}|_{\p\Om\times [0,T_\ep)}=0$. By the equation
	\[\frac{\p u_\ep}{\p t}=\ep\tau(u_\ep)+u_\ep\times \tau(u_\ep)\]
	we can show that
	\begin{align*}
	\text{LHS of \eqref{eq}}=&\ep\int_0^{T}\int_{\Om}\<\De \tau(u_\ep), \phi\>dxdt+\int_0^{T}\int_{\Om}\<\De (u_\ep\times \tau(u_\ep)), \phi\>dxdt\\
	=&\ep\int_0^{T}\int_{\Om}\<\De \tau(u_\ep), \phi\>dxdt\\
	&+\int_0^{T}\int_{\Om}\<\De u_\ep\times \tau(u_\ep)+2\n u_\ep\times \n \tau(u_\ep)+u_\ep\times \De \tau(u_\ep), \phi\>dxdt
	\end{align*}
	and
	\begin{align*}
	\text{RHS of \eqref{eq}}=&-\ep \int_0^{T}\int_{\Om}\<\n \tau(u_\ep),\n \phi\>dxdt-\int_0^{T}\int_{\Om}\<\n u_\ep \times \tau(u_\ep)+u_\ep \times \n\tau(u_\ep),\n \phi\>dxdt\\
	=&+\ep \int_0^{T}\int_{\Om}\<\De \tau(u_\ep),\phi\>dxdt-\ep \int_0^{T}\int_{\p \Om}\<\frac{\p \tau(u_\ep)}{\p \nu},\phi\>dxdt\\
	&-\int_0^{T}\int_{\Om}\<\n u_\ep ,\n(\tau(u_\ep)\times \phi)\>dxdt+\int_0^{T}\int_{\Om}\<\n u_\ep ,\n \tau(u_\ep)\times \phi \>dxdt\\
	&+\int_0^{T}\int_{\Om}\<\mbox{div}(u_\ep\times \n \tau(u_\ep)) ,\phi)\>dxdt-\int_0^{T}\int_{\p\Om}\<u_\ep\times
	\frac{\p\tau(u_\ep)}{\p \nu} ,\phi)\>dxdt\\
	=&\ep \int_0^{T}\int_{\Om}\<\De \tau(u_\ep),\phi\>dxdt-\int_0^{T}\int_{\p \Om}\<\ep \frac{\p \tau(u_\ep)}{\p \nu}+u_\ep\times
	\frac{\p\tau(u_\ep)}{\p \nu},\phi\>dxdt\\
	&+\int_0^{T}\int_{\Om}\<\De u_\ep\times \tau(u_\ep)+2\n u_\ep\times \n \tau(u_\ep)+u_\ep\times \De \tau(u_\ep), \phi\>dxdt.\\
	\end{align*}
Then, the fact ``$LHS=RHS$'' leads to
	\[\int_0^{T}\int_{\p \Om}\<\ep \frac{\p \tau(u_\ep)}{\p \nu}+u_\ep\times
	\frac{\p\tau(u_\ep)}{\p \nu},\phi\>dxdt=0\]
in the sense of trace.
\end{proof}

\begin{lem}\label{equi-norm}
The solution $u_\ep$ has the following properties:
\begin{itemize}
\item[$(1)$] $\De u_\ep=\frac{1}{1+\ep^2}(\ep \frac{\p u_\ep}{\p t}-u_\ep\times  \frac{\p u_\ep}{\p t})-|\n u_\ep|^2u_\ep$ for a.e. $(x,t)\in \Om\times[0,T_\ep)$;
\item[$(2)$] There exists a constant $C$ such that there holds
\[\int_{\Om}|\n \De u_\ep|^2dx\leq C(1+\frac{1}{1+\ep^2})(1+\norm{u_\ep}^2_{H^2}+\norm{\frac{\p u_\ep}{\p t}}^2_{H^1})^3,\]
and hence we have
\[\norm{u_\ep}^2_{H^3}\leq C(1+\norm{u_\ep}^2_{H^2}+\norm{\frac{\p u_\ep}{\p t}}^2_{H^1})^3.\]
\end{itemize}
\end{lem}
\begin{proof}
Since $u_\ep$ satisfies the perturbation equation \eqref{PS-eq}:
\[\frac{\p u_\ep}{\p t}=\ep \tau(u)+u_\ep\times \tau(u_\ep),\]
a direct calculation shows
\[\tau(u_\ep)=\frac{1}{1+\ep^2}(\ep \frac{\p u_\ep}{\p t}-u_\ep\times\frac{\p u_\ep}{\p t}).\]
Thus, the proof of $(1)$ is finished. And hence, there holds
\begin{align*}
\n \De u_\ep=&\frac{1}{1+\ep^2}(\ep \n \frac{\p u_\ep}{\p t}-\n u_\ep\times  \frac{\p u_\ep}{\p t}-u_\ep\times \n \frac{\p u_\ep}{\p t})\\
&-|\n u_\ep|^2\n  u_\ep-2\<\n^2 u_\ep, \n u_\ep\>u_\ep.
\end{align*}
Immediately it follows
\begin{align*}
\int_{\Om}|\n \De u_\ep|^2dx\leq& \frac{C}{(1+\ep^2)^2}\int_{\Om}(1+\ep^2)|\n \frac{\p u_\ep}{\p t}|^2+|\n u_\ep|^2|\frac{\p u_\ep}{\p t}|^2dx\\
&+C\int_{\Om}|\n u_\ep|^6+|\n^2 u_\ep|^2|\n u_\ep|^2dx\\
=&\frac{C}{1+\ep^2}(\norm{\n\frac{\p u_\ep}{\p t}}^2_{L^2}+I)+C(\norm{u_\ep}^6_{H^2}+II).
\end{align*}
Here,
\begin{align*}
I=&\int_{\Om}|\n u_\ep|^2|\frac{\p u_\ep}{\p t}|^2dx\\
\leq &C\norm{\n u_\ep}^2_{L^4}\norm{\frac{\p u_\ep}{\p t}}^2_{L^4}\leq C (\norm{u_\ep}^4_{H^2}+\norm{\frac{\p u_\ep}{\p t}}^4_{H^1}),
\end{align*}
and
\begin{align*}
II=&\int_{\Om}|\n^2 u_\ep|^2|\n u_\ep|^2dx\\
\leq &\norm{\n^2 u_\ep}^2_{L^3}\norm{\n u_\ep}^2_{L^6}\leq \norm{\n ^2u_\ep}_{L^2}\norm{\n^2 u_\ep}_{L^6}\norm{\n u_\ep}^2_{L^6}\\
\leq &C\norm{u_\ep}^3_{H^2}(\norm{u_\ep}_{H^2}+\norm{\n \De u_\ep}_{L^2})\\
\leq &C(\de)(\norm{u_\ep}^4_{H^2}+\norm{u_\ep}^6_{H^2})+\de\norm{\n\De u_\ep}^2_{L^2}.
\end{align*}
To get the above estimate of  $II$,  we have used the interpolation inequality
\[\norm{\n^2 u_\ep}_{L^3}\leq \norm{\n ^2u_\ep}^\frac{1}{2}_{L^2}\norm{\n^2 u_\ep}^\frac{1}{2}_{L^6}\]
and Lemma \ref{eq-norm}. Thus, in view of the estimates of $I$ and $II$, the estimate in $(2)$ has been established.
\end{proof}

Now, we are in the position to show the uniform $H^2$-estimates of solution $u_\ep$. By choosing $\frac{\p u_\ep}{\p t}$ as a test function to equation \eqref{w-eq2}, we obtain
\begin{equation*}
\begin{aligned}
&\int_{\Om}\<\frac{\p^2 u_\ep}{\p t^2},\frac{\p u_\ep}{\p t}\>dx+(1-\ep^2)\int_{\Om}\<\De (\De u_\ep+|\n u_\ep|^2u_\ep),\frac{\p u_\ep}{\p t}\>dx\\
=&\ep\left\{\int_{\Om}\<2\De(u_\ep\times \De u_\ep)+\frac{\p}{\p t}(|\n u_\ep|^2u_\ep)-2\textnormal{\mbox{div}}(\n u_\ep\dot{\times} \n^2 u_\ep)+u_\ep\times \De(|\n u_\ep|^2u_\ep),\frac{\p u_\ep}{\p t}\>dx\right\}\\
&+\ep\left\{\int_{\Om}\<|\n u_\ep|^2u_{\ep}\times \De u_\ep, \frac{\p u_\ep}{\p t}\> dx\right\}\\
&+\int_{\Om}\<\De(|\n u_\ep|^2 u_\ep)-2\textnormal{\mbox{div}}^2((\n u_\ep\dot{\otimes}\n u_\ep))u_\ep-2\<\De u_\ep, \n u_\ep\>\cdot \n u_\ep-2\<\n |\n u_\ep|^2,\n u_\ep\>,\frac{\p u_\ep}{\p t}\>dx\\
&-\int_{\Om}\<|\n u_\ep|^2\De u_\ep,\frac{\p u_\ep}{\p t}\>dx.
\end{aligned}
\end{equation*}
For the sake of convenience, we rewrite the above identity as the following
\begin{equation}\label{H^1-inq}
\begin{aligned}
&\int_{\Om}\<\frac{\p^2 u_\ep}{\p t^2},\frac{\p u_\ep}{\p t}\>dx+(1-\ep^2)\int_{\Om}\<\De (\De u_\ep+|\n u_\ep|^2u_\ep),\frac{\p u_\ep}{\p t}\>dx\\
=&\ep(I_1+I_2+I_3+I_4+I_5)+(II_1+II_2+II_3+II_4+II_5).
\end{aligned}
\end{equation}
First of all, we estimate the left hand side of the above identity:
\[LHS=\frac{1}{2}\frac{\p}{\p t}\int_{\Om}|\frac{\p u_\ep}{\p t}|^2dx+\frac{1-\ep^2}{2}\frac{\p}{\p t}\int_{\Om}|\De u_\ep|^2dx-(1-\ep^2)\int_{\Om}\<\n (|\n u_\ep|^2u_\ep), \n \frac{\p u_\ep}{\p t}\>dx.\]
Here we have used the compatibility boundary condition in Lemma \ref{comp-bdy}. For the last term in the above identity we have the following estimate:
\begin{equation}\label{H^1-inq1}
\begin{aligned}
(1-\ep^2)\left|\int_{\Om}\<\n(|\n u_\ep|^2u_\ep), \n\frac{\p u_\ep}{\p t}\>dx\right|\leq & C\int_{\Om}|\n u_\ep|^6+|\n u_\ep|^2|\n^2 u_\ep|^2+|\n\frac{\p u_\ep}{\p t}|^2dx\\
\leq&C(\norm{u_\ep}^6_{H^2}+\norm{u_\ep}^2_{H^3}\norm{u_\ep}^2_{H^2}+\norm{\n\frac{\p u_\ep}{\p t}}^2_{L^2})\\
\leq &C(1+\norm{u_\ep}^2_{H^2}+\norm{\frac{\p u_\ep}{\p t}}^2_{H^1})^4,
\end{aligned}
\end{equation}
where we have used the estimates in Lemma \ref{equi-norm}.
	
Next, we estimate the nine terms on the right hand side step by steps.
\begin{equation}\label{H^1-inq2}
\begin{aligned}
I_1=&2\int_{\Om}\<\De (u_\ep\times \De u_\ep ),\frac{\p u_\ep}{\p t}\>dx=-2\int_{\Om}\<\De (\ep(\De u_\ep +|\n u_\ep|^2u_\ep)-\frac{\p u_\ep}{\p t}),\frac{\p u_\ep}{\p t}\>dx\\
=&-2\ep \int_{\Om}\<\De u_\ep +|\n u_\ep|^2u_\ep,\De\frac{\p u_\ep}{\p t}\>dx-2 \int_{\Om}\<\n \frac{\p u_\ep}{\p t},\n \frac{\p u_\ep}{\p t}\>dx\\
=&-\ep\frac{\p}{\p t}\int_{\Om}|\De u_\ep|^2dx+2\ep\int_{\Om}\< \n(|\n u_\ep|^2u_\ep),\n\frac{\p u_\ep}{\p t}\>dx-2\norm{\n\frac{\p u_\ep}{\p t}}^2_{L^2}\\
\leq&-\ep\frac{\p}{\p t}\int_{\Om}|\De u_\ep|^2dx+C\ep(\norm{u_\ep}^2_{H^2}+\norm{\frac{\p u_\ep}{\p t}}^2_{H^1})^4-(2-\ep)\norm{\n\frac{\p u_\ep}{\p t}}^2_{L^2}.
\end{aligned}
\end{equation}
Here, we have used the compatibility conditions in Line 2 and Line 3 of the above inequality \eqref{H^1-inq2}, and applied Lemma \ref{equi-norm} in the last line.
	
\begin{equation}\label{H^1-inq3}
\begin{aligned}
|I_2|=&\left|\int_{\Om}\<\frac{\p}{\p t}(|\n u_\ep|^2u_\ep),\frac{\p u_\ep}{\p t}\>dx\right|=\int_{\Om}|\n u_\ep|^2|\frac{\p u_\ep}{\p t}|^2dx\\
\leq &\norm{u_\ep}^2_{H^2}\norm{\frac{\p u_\ep}{\p t}}^2_{H^1}\leq C(\norm{u_\ep}^2_{H^2}+\norm{\frac{\p u_\ep}{\p t}}^2_{H^1})^2.
\end{aligned}
\end{equation}
Here we have used the fact $\<u_\ep,\frac{\p u_\ep}{\p t}\>=0$.
	
\begin{equation}\label{H^1-inq4}
\begin{aligned}
|I_3|=&\left|\int_{\Om}\<\mbox{div}(\n u_\ep\dot{\times} \n^2 u_\ep),\frac{\p u_\ep}{\p t}\>dx\right|= \left|\int_{\Om}\sum_{j=1}^3\<\p_j u_\ep\times \p_j\De u_\ep),\frac{\p u_\ep}{\p t}\>dx\right|\\
\leq &C \int_{\Om}|\n u_\ep||\n\De u_\ep||\frac{\p u_\ep}{\p t}|dx\leq C\norm{\n u_\ep}_{L^6}\norm{\n\De u_\ep}_{L^2}\norm{\frac{\p u_\ep}{\p t}}_{L^3}\\
\leq &C(\norm{\n\De u_\ep}^2_{L^2}+\norm{u_\ep}^2_{H^2}\norm{\frac{\p u_\ep}{\p t}}^2_{H^1})\leq C(1+\norm{u_\ep}^2_{H^2}+\norm{\frac{\p u_\ep}{\p t}}^2_{H^1})^3.
\end{aligned}
\end{equation}
	
\begin{equation}\label{H^1-inq5}
\begin{aligned}
|I_4|=&|\int_{\Om}\<u_\ep\times \De(|\n u_\ep|^2u_\ep),\frac{\p u_\ep}{\p t}\>dx|\leq C\int_{\Om}|\n^2 u_\ep||\n u_\ep|^2|\frac{\p u_\ep}{\p t}|dx\\
\leq &C\norm{\n^2 u_{\ep}}_{L^2}\norm{\n u_\ep}^2_{L^6}\norm{\frac{\p u_\ep}{\p t}}_{L^6}\leq C(1+\norm{u_\ep}^6_{H^2}+\norm{\frac{\p u_\ep}{\p t}}^2_{H^1}).
\end{aligned}
\end{equation}
	
For the term $$|I_5|=|\int_{\Om}\<|\n u_\ep|^2u_\ep\times \De u_\ep,\frac{\p u_\ep}{\p t}\>dx|$$ and the term $$|II_1|=|\int_{\Om}\<\De(|\n u_\ep|^2u_\ep),\frac{\p u_\ep}{\p t}\>dx|,$$
we can derive the same estimates as that of term $I_4$.

On the other hand, a simple calculation leads to
\begin{equation}\label{H^1-inq6}
\begin{aligned}
|II_2+II_3+II_4+II_5|\leq&2\left|\int_{\Om}\<\mbox{div}^2((\n u_\ep\dot{\otimes}\n u_\ep))u_\ep,\frac{\p u_\ep}{\p t}\>dx\right|\\
&+2\left|\int_{\Om}\<\<\De u_\ep, \n u_\ep\>\cdot \n u_\ep,\frac{\p u_\ep}{\p t}\>dx\right|\\
&+2\left|\int_{\Om}\<(\n|\n u_\ep|^2\cdot\n u_\ep),\frac{\p u_\ep}{\p t}\>dx\right|\\
&+\left|\int_{\Om}\<|\n u_\ep|^2\De u_\ep,\frac{\p u_\ep}{\p t}\>dx\right|\\
\leq &C\int_{\Om}|\n u_\ep|^2|\n^2 u_\ep||\frac{\p u_\ep}{\p t}|dx\\
\leq & C(1+\norm{u_\ep}^6_{H^2}+\norm{\frac{\p u_\ep}{\p t}}^2_{H^1}).
\end{aligned}
\end{equation}
Therefore, by combining inequalities \eqref{H^1-inq1}-\eqref{H^1-inq6} with equation \eqref{H^1-inq}, we get the uniform $H^2$-estimates of $u_\ep$ as follows:
\begin{equation}\label{H^2}
\begin{aligned}
\frac{1}{2}\frac{\p}{\p t}\int_{\Om}|\frac{\p u_\ep}{\p t}|^2dx+\frac{1+\ep^2}{2}\frac{\p}{\p t}\int_{\Om}|\De u_\ep|^2dx\leq C(1+\ep)(1+\norm{u_\ep}^2_{H^2}+\norm{\frac{\p u_\ep}{\p t}}^2_{H^1})^4,
\end{aligned}
\end{equation}
where the constant $C$ is independent of $\ep$ ($0<\ep<1$) and $u_\ep$.

\subsection{Uniform $H^3$-estimates}
In this subsection, we show the uniform $H^3$-estimates of solution $u_\ep$. By a similar argument with that in the above subsection, we choose $-\De \frac{\p u_\ep}{\p t}$ as a test function of equation \eqref{w-eq1}(or \eqref{w-eq2}). However, it seems that it is not easy to get energy estimates directly, since the regularity of $\frac{\p^2 u_\ep}{\p t^2}\in L^2(\Om\times[0,T])$ and $\De^2u_\ep\in L^2(\Om\times[0,T])$ for all $T<T_\ep$ established in Theorem \ref{uin-ex-PS} is not high enough, and hence integration by parts does not make sense. To proceed, first of all, we enhance the local regularity of solution such that $\frac{\p u_\ep}{\p t}\in L^2_{loc}((0,T_\ep), H^3(\Om))$ by using the $L^2$-estimates of parabolic equation, since $\frac{\p u_\ep}{\p t}$ satisfies equation \eqref{w-eq1}:
\begin{equation}
\begin{cases}
\frac{\p v}{\p t}-\ep\De v-u_\ep\times\De v=f(u_\ep, \n v, v),\\[1ex]
v(x,0)=\frac{\p u_\ep}{\p t}|_{t=0},\quad \frac{\p v}{\p \nu}|_{\p\Om\times (0, T_\ep)}=0,
\end{cases}
\end{equation}
which is linear and uniformly parabolic equation when $\ep>0$. Here
\[f(u_\ep, \n v, v)=v\times \De u_\ep+2\ep\<\n u_\ep,\n v\>u_\ep+\ep |\n u_\ep|^2v\in L^{2}_{loc}([0, T_\ep), H^1(\Om)),\]
and
\[\frac{\p u_\ep}{\p t}|_{t=0}=\ep(\De u_0+|\n u_0|^2u_0)+u_0\times \De u_0.\]

Moreover, since $u_\ep\in L^\infty([0,T], H^3(\Om))$ and $\frac{\p u_\ep}{\p t}\in L^2([0,T], L^2(\Om))$ for all $T<T_\ep$, then $(2)$ of Lemma \ref{A-S} implies
\[u_{\ep}\in C^0([0,T], H^2(\Om)).\]
Immediately it follows that $u_\ep\in C^0(\Om\times [0,T])$. Indeed, for any $(x,t)$ and $(x_0,t_0)$ we have
\begin{equation}
\begin{aligned}
|u_\ep(x,t)-u_\ep(x_0,t_0)|\leq&|u_\ep(x,t)-u_\ep(x_0,t)|+|u_\ep(x_0,t)-u_\ep(x_0,t_0)|\\
\leq &\sup_{\Om}|\n u_\ep|(\cdot,t)|x-x_0|+|u_\ep(x_0,t)-u_\ep(x_0,t_0)|\\
\leq &C\norm{u_\ep}_{L^\infty([0,T], H^3(\Om))}|x-x_0|+C\norm{u_\ep(\cdot, t)-u_\ep(\cdot, t_0)}_{H^2(\Om)}.
\end{aligned}
\end{equation}
and hence it implies $u_\ep\in C^0(\Om\times [0,T])$. Hence, the $L^2$-theory of parabolic equation tells us that
\[\frac{\p^2 u_\ep}{\p t^2}\in L^{2}_{loc}((0,T_\ep), H^1(\Om)),\]
which guarantees the integration by parts in the process of energy estimates makes sense. For the fluency and shortness of this article, we give the above process of improving regularity of $\frac{\p u_\ep}{\p t}$ in Appendix \ref{ref-A}.

Now, we turn back to show the uniform $H^3$-estimates of $u_\ep$. To this end, we choose $-\De \frac{\p u_\ep}{\p t}$ as a test function of \eqref{w-eq1} and take a simple calculation to obtain
	
\begin{equation}
\begin{aligned}
&\frac{1}{2}\(\int_{\Om}|\n \frac{\p u_\ep(T)}{\p t}|^2dx-\int_{\Om}|\n \frac{\p u_\ep(0)}{\p t}|^2dx\)+\ep \int_0^T\int_{\Om}|\De \frac{\p u_\ep}{\p t}|^2dxdt\\
=&-\int_0^T\int_{\Om}\<\frac{\p u_\ep}{\p t}\times \De u_\ep, \De \frac{\p u_\ep}{\p t}\>dxdt-\ep\int_0^T\int_{\Om}\<\frac{\p}{\p t}(|\n u_\ep|^2u_\ep),\De \frac{\p u_\ep}{\p t}\>dxdt\\
=& V+W,
\end{aligned}
\end{equation}
since $\frac{\p u_\ep}{\p t} \in C^0([0,T], H^1(\Om))$ for $T<T_\ep$.
	
In order to get the desired estimates of $\frac{\p u_\ep}{\p t}$, we turn to estimating the terms $V$ and $W$.
\begin{equation}
\begin{aligned}
|V|=&\left|\int_0^T\int_{\Om}\<\frac{\p u_\ep}{\p t}\times \n \De u_\ep, \n \frac{\p u_\ep}{\p t}\>dxdt\right|\\
\leq &\frac{1}{1+\ep^2}\left|\int_0^T\int_{\Om}\<\frac{\p u_\ep}{\p t}\times(\n u_\ep \times \frac{\p u_\ep}{\p t}+u_\ep\times \n \frac{\p u_\ep}{\p t}), \n \frac{\p u_\ep}{\p t}\>dxdt\right|\\
&+\left|\int_0^T\int_{\Om}\<\frac{\p u_\ep}{\p t}\times\n(|\n u_\ep|^2u_\ep),\n \frac{\p u_\ep}{\p t}\>dxdt\right|.
\end{aligned}
\end{equation}
Here we have used the fact:
\[\De u_\ep=\frac{1}{1+\ep^2}(\ep \frac{\p u_\ep}{\p t}-u_\ep\times  \frac{\p u_\ep}{\p t})-|\n u_\ep|^2u_\ep.\]
Using again the Lagrangian formula, there holds
\begin{align*}
&\frac{\p u_\ep}{\p t}\times(u_\ep\times \n \frac{\p u_\ep}{\p t})\\
=&\<\n \frac{\p u_\ep}{\p t},\frac{\p u_\ep}{\p t}\>u_\ep-\<\frac{\p u_\ep}{\p t},u_\ep\>\n \frac{\p u_\ep}{\p t}\\
=&\<\n \frac{\p u_\ep}{\p t},\frac{\p u_\ep}{\p t}\>u_\ep.
\end{align*}
Noting $|u_\ep|=1$ implies
\[\<\n \frac{\p u_\ep}{\p t}, u_\ep\>=-\<\frac{\p u_\ep}{\p t}, \n u_\ep\>,\]
we have that
\begin{align*}
\<\frac{\p u_\ep}{\p t}\times(u_\ep\times \n \frac{\p u_\ep}{\p t}),\n \frac{u_\ep}{\p t}\>=&\<\n \frac{\p u_\ep}{\p t},\frac{\p u_\ep}{\p t}\>\cdot\<u_\ep,\n \frac{\p u_\ep}{\p t}\>\\
=&-\<\n \frac{\p u_\ep}{\p t},\frac{\p u_\ep}{\p t}\>\cdot\<\n u_\ep, \frac{\p u_\ep}{\p t}\>.
\end{align*}
Then, it follows that there holds
\begin{equation}
\begin{aligned}
|V|\leq& \frac{2}{1+\ep^2}\int_0^T\int_{\Om}|\frac{\p u_\ep}{\p t}|^2|\n u_\ep||\n \frac{\p u_\ep}{\p t}|dxdt\\
&+2\int_0^T\int_{\Om}(|\n u_\ep|^3+|\n^2u_\ep||\n u_\ep|)|\frac{\p u_\ep}{\p t}||\n \frac{\p u_\ep}{\p t}|dxdt\\
\leq &\frac{C}{1+\ep^2}\int_0^T\norm{\frac{\p u_\ep}{\p t}}^2_{L^6}\norm{\n u_\ep}_{L^6}\norm{\n \frac{\p u_\ep}{\p t}}_{L^2}dt\\
&+C\int_0^T(\norm{\n u_\ep}^3_{L^\infty}\norm{\frac{\p u_\ep}{\p t}}_{L^2}+\norm{\n u_\ep}_{L^\infty}\norm{\n^2 u_\ep}_{L^3}\norm{\frac{\p u_\ep}{\p t}}_{L^6})\norm{\n \frac{\p u_\ep}{\p t}}_{L^2}dt\\
\leq &C\int_0^T(\norm{u_\ep}^2_{H^2}+\norm{\frac{\p u_\ep}{\p t}}^6_{H^1})dt+C\int_0^T(\norm{u_\ep}^3_{H^3}+\norm{u_\ep}^2_{H^3})\norm{\frac{\p u_\ep}{\p t}}^2_{H^1}dt\\
\leq &C\int_0^T(1+\norm{u_\ep}^2_{H^2}+\norm{\frac{\p u_\ep}{\p t}}^2_{H^1})^6dt.
\end{aligned}
\end{equation}
Here we have used the estimate in Lemma \eqref{equi-norm}, and used the Sobolove embedding inquality
\[\norm{\n u_\ep}_{L^\infty}\leq C\norm{u_\ep}_{H^3}\leq C(1+\norm{u_\ep}^2_{H^2}+\norm{\frac{\p u_\ep}{\p t}}^2_{H^1})^{3/2}.\]
	
Next, we give the estimates of $W$ as follows.
\begin{equation}
\begin{aligned}
|W|\leq&\ep\int_0^T\int_{\Om}|\frac{\p}{\p t}(|\n u_\ep|^2u_\ep)||\De \frac{\p u_\ep}{\p t}|dxdt\\
\leq &C\ep\int_0^T\int_{\Om}|\frac{\p u_\ep}{\p t}|^2|\n u_\ep|^4+|\n u_\ep|^2|\n \frac{\p u_\ep}{\p t}|^2dxdt+\frac{\ep}{2}\int_0^T\int_{\Om}|\De \frac{\p u_\ep}{\p t}|^2dxdt\\
\leq& C\ep\int_0^T(\norm{u_\ep}^2_{H^2}+\norm{\frac{\p u_\ep}{\p t}}^2_{H^1})^3dt+\frac{\ep}{2}\int_0^T\int_{\Om}|\De \frac{\p u_\ep}{\p t}|^2dxdt.
\end{aligned}
\end{equation}
Here we have used the fact:
\[\int_{\Om}|\frac{\p u_\ep}{\p t}|^2|\n u_\ep|^4dx\leq \norm{\frac{\p u_\ep}{\p t}}^2_{L^6}\norm{\n u_\ep}^4_{L^6},\]
and
\[\norm{\n u_\ep}_{L^\infty}\leq C\norm{u_\ep}_{H^3}\leq C(\norm{u_\ep}^2_{H^2}+\norm{\frac{\p u_\ep}{\p t}^2_{H^1}})^{3/2}.\]
	
Therefore, we concludes
\begin{equation}\label{H^3}
\begin{aligned}
&\frac{1}{2}\int_{\Om}|\n \frac{\p u_\ep}{\p t}|^2dx|_{t=T}+\frac{\ep}{2} \int_0^T\int_{\Om}|\De \frac{\p u_\ep}{\p t}|^2dxdt\\
\leq &\frac{1}{2}\int_{\Om}|\n \frac{\p u_\ep}{\p t}|^2dx|_{t=0}+C\int_0^T(1+\norm{u_\ep}^2_{H^2}+\norm{\frac{\p u_\ep}{\p t}}^2_{H^1})^6dt.
\end{aligned}
\end{equation}
	
By combining the estimates \eqref{H^1},\eqref{H^2} with \eqref{H^3}, we can derive that there holds
\begin{equation}\label{es-u}
\begin{aligned}
&(\norm{u_{\ep}(T)}^2_{H^2}+\norm{\frac{\p u_{\ep}(T)}{\p t}}^2_{H^1})+\ep\int_{0}^{T}\int_{\Om}|\De \frac{\p u_\ep}{\p t}|^2dxdt\\
\leq &(\norm{u_0}_{H^2}+\norm{\frac{\p u_\ep}{\p t}}^2_{H^1}|_{t=0})+C\int_0^T(1+\norm{u_{\ep}}^2_{H^2}+\norm{\frac{\p u_{\ep}}{\p t}}^2_{H^1})^6dt
\end{aligned}
\end{equation}
for all $0<T<T_\ep$.
	
By using the Gronwall-type inequality in Lemma \ref{Gron-inq}, the desired estimates of approximated solution $u_\ep$ are derived from \eqref{es-u}, we formulated the estimates in the following lemma.
	
\begin{lem}\label{U-es: u_e}
There exists a positive number $T_0$ and a constant $C(T_0)$, which depends only on $\norm{u_0}_{H^3}$, such that for all $T<\min\{T_0, T_\ep\}$, the solution  $u_\ep$ obtained in Theorem \eqref{uin-ex-PS} satisfies the following uniform estimate:
\[\sup_{0<t\leq T}(\norm{u_{\ep}}^2_{H^2}+\norm{\frac{\p u_{\ep}}{\p t}}^2_{H^1})\leq C(T_0).\]
\end{lem}
	
\begin{proof}
Let $$y(t)=(\norm{u_{\ep}}^2_{H^2}+\norm{\frac{\p u_{\ep}}{\p t}}^2_{H^1})(t).$$
Since $u_\ep \in L^2([0,T], H^4(\Om))$, $\frac{\p u_\ep}{\p t}\in L^2([0,T], H^2(\Om))$ and $\frac{\p^2 u_\ep}{\p t^2}\in L^2([0,T], L^2(\Om))$, the embedding Lemma \ref{C^0-em} implies
\[y(t)=(\norm{u_{\ep}}^2_{H^2}+\norm{\frac{\p u_{\ep}}{\p t}}^2_{H^1})(t)\in C^0([0,T_\ep)),\]
and hence
\[\frac{\p u_\ep}{\p t}|_{t=0}=\ep (\De u_0+|\n u_0|^2u_0)+u_0\times \De u_0,\]
and
		\[\n\frac{\p u_\ep}{\p t}|_{t=0}=\n(\ep\De u_0+\ep|\n u_0|^2u_0+u_0\times \De u_0).\]
Thus, a direct calculation shows
\[\norm{\frac{\p u_\ep}{\p t}}^2_{H^1}|_{t=0}\leq C(1+\ep^2)(1+\norm{u_0}^6_{H^3}).\]
\end{proof}

Let $y_0=C(1+\norm{u_0}^6_{H^3})$ and $f(y)=C(1+y)^6$. Then, the function $y(t)$ satisfies the following inequality
\[y(t)\leq y_0+\int_0^tf(y(s))ds.\]
Then, the Gronwall-type inequality in Lemma \ref{Gron-inq} implies that there exists a positive number $T_0>0$ depending only on $y_0$ and a constant $C(T_0)$ such that for all $T<\min\{T_0, T_\ep\}$ there holds
\[\sup_{0<t<T_0}y(t)\leq C(T_0).\]
Thus, the proof is completed.
\medskip
	
\section{regular solution to the Schr\"{o}dinger flow}\label{s2}
In this section, we prove the local existence of strong solutions to \eqref{S-eq} in Theorem \ref{thm1}. To this end, we need to give an uniform lower bound of existence times $T_\ep$ and the compactness of the approximation solution $u_\ep$ to \eqref{PS-eq}. Consequently, we can claim that the limit map $u$ of sequence $\{u_\ep\}$ is a strong solution to \eqref{S-eq}.
\begin{thm}
There exists a positive time $T_0$ depending only on $\norm{u_0}_{H^3(\Om)}$ such that the equation \eqref{S-eq} admits a local regular solution on $[0,T_0]$, which satisfies
\[u\in L^\infty([0,T_0], H^3(\Om)) \quad \text{and}\quad \frac{\p u}{\p t}\in L^\infty([0,T_0],H^1(\Om)).\]
\end{thm}
\begin{proof}
We divide the proof into three steps.
\medskip
		
\noindent\emph{Step 1: The uniform positive lower bound of $T_\ep$.}
\medskip

We claim $T_0<T_\ep$, where $T_0$ was obtained in Lemma \ref{U-es: u_e}. Suppose that $T_0\geq T_\ep$, then Lemma \ref{U-es: u_e}  and Lemma \ref{equi-norm} tell us that there holds
\[\sup_{0<t<T_\ep}(\norm{u_{\ep}}^2_{H^3}+\norm{\frac{\p u_{\ep}}{\p t}}^2_{H^1})\leq C(T_0).\]
Since $v=\frac{\p u_\ep}{\p t}$  satisfies the following equation
\begin{equation*}
\begin{cases}
\frac{\p v}{\p t}-\ep\De v-u_\ep\times\De v=f(u_\ep, \n v, v),\\[1ex]
v(x,0)=\frac{\p u_\ep}{\p t}|_{t=0},\quad \frac{\p v}{\p t}|_{\p\Om\times (0, T_\ep)}=0,
\end{cases}
\end{equation*}
where
\[f(u_\ep, \n v, v)=v\times \De u_\ep+2\ep\<\n u_\ep,\n v\>u_\ep+\ep |\n u_\ep|^2v,\]
and we can easily to verify that the homogeneous term $f\in L^2([0,T_\ep], L^2(\Om))$, it follows from the $L^2$-estimates in Theorem \ref{reg-es} in Appendix A that
\[\frac{\p u_\ep}{\p t}\in L^2([\de, T_\ep], H^2(\Om)),\]
and
\[\frac{\p ^2u_\ep}{\p t^2}\in L^2([\de, T_\ep], L^2(\Om))\]
for any small $0<\de<T_\ep$. By equation \eqref{PS-eq}, we know that there holds
\[\sup_{0<t<T_\ep}(\norm{u_{\ep}}^2_{H^3}+\norm{\frac{\p u_{\ep}}{\p t}}^2_{H^1})+\int_{0}^{T_\ep}\int_{\Om}(\norm{u_\ep}^2_{H^4(\Om)}+\norm{\frac{\p u_\ep}{\p t}}^2_{H^2(\Om)}+\norm{\frac{\p^2 u_\ep}{\p t^2}}^2_{L^2(\Om)})dt<\infty.\]
Thus, it implies that the existence interval $[0, T_\ep]$ can be extended, which is a contradiction with the definition of $T_\ep$. Therefore, we have
\[T_0< T_\ep.\]
\medskip
		
\noindent\emph{Step 2: The compactness of $u_\ep$.}\
\medskip
		
Lemma \ref{U-es: u_e} tells us there exists a constant $C(T_0)$, which is independent of $\ep$, such that there holds true
\[\sup_{0<t\leq T_0}\norm{u_\ep}_{H^3}+\sup_{0<t\leq T_0}\norm{\frac{\p u_\ep}{\p t}}_{H^1}\leq C(T_0).\]
Without loss of generality, we assume that there exists a map in $u\in L^\infty([0,T_0], H^3(\Om))$ such that
\[u_\ep\rightharpoonup u\quad \text{weak* in}\quad u\in L^\infty([0,T_0], H^3(\Om)),\]
\[\frac{\p u_\ep}{\p t}\rightharpoonup \frac{\p u}{\p t}\quad \text{weakly* in}\quad L^\infty([0,T_0], H^1(\Om)),\]
\[\frac{\p u_\ep}{\p t}\rightharpoonup \frac{\p u}{\p t}\quad \text{weakly in}\quad L^2([0,T_0], H^1(\Om)).\]
Let $X=H^3(\Om)$, $B=H^2(\Om)$ and $Y=L^2$, by Lemma \ref{A-S} we have
\[u_\ep\to u \quad \text{strongly in}\quad L^\infty([0,T_0], H^2(\Om)),\]
and hence
\[u_\ep\to u \quad  \text{a.e. (x,t)} \in \Om\times[0,T_0].\]
It follows immediately that $|u|=1$ for a.e. $(x,t) \in [0,T_0]\times\Om$.
\medskip
		
\noindent\emph{Step 3: The regular solution to \eqref{S-eq}.}\
\medskip
		
Since $u_\ep$ is a strong solution to \eqref{PS-eq}, there holds
\[\int_{0}^{T_0}\int_{\Om}\<\frac{\p u_\ep}{\p t},\phi\>dxdt-\ep\int_{0}^{T_0}\int_{\Om}\<\De u_\ep+|\n u_\ep|^2 u_\ep,\phi\>dxdt=\int_{0}^{T_0}\int_{\Om}\<u_\ep\times \De u_\ep,\phi\>dxdt,\]
for all $\phi\in C^\infty(\bar{\Om}\times[0,T])$. By using the convergence results on $u_\ep$ in Step $2$, it is easy to show directly $u$ is a strong solution to \eqref{S-eq} by letting $\ep\to 0$.
		
To complete the proof, we need to check $u$ satisfies the Neumann boundary condition, that is $\frac{\p u}{\p \nu}|_{\p \Om\times[0,T_0]}=0$. Since for any $\xi\in C^{\infty}(\bar{\Om}\times[0,T_0])$, there holds
$$\int_{0}^{T_0}\int_{\Om}\<\De u_{\ep},\xi\>dxdt=-\int_{0}^{T_0}\int_{\Om}\<\n u_\ep,\n \xi\>dxdt.$$
Letting $n\to \infty$, we have
$$\int_{0}^{T_0}\int_{\Om}\<\De u,\xi\>dxdt=-\int_{0}^{T_0}\int_{\Om}\<\n u,\n \xi\>dxdt.$$
This means that there holds true $$\frac{\p u}{\p \nu}|_{\p \Om\times [0,T_0]}=0.$$
\end{proof}
	
\section{The uniqueness of solutions to the Schr\"{o}dinger flow}\label{s3}
In this section, we show the uniqueness of the solution $u$ to \eqref{S-eq}(or \eqref{S-eq1}) in the space
\[\S=\{u\, |u\in L^\infty([0,T], H^3(\Om))\,\, \text{and}\,\, \frac{\p u}{\p t}\in L^\infty([0,T], H^1(\Om))\}.\]
We will only give the sketches of the proof for the uniqueness in Theorem \ref{thm1}, since the arguments go almost the same as that of the proof of uniqueness for the Schr\"{o}dinger flows from a general compact Riemannian manifold to a K\"{a}hler manifold in \cite{S-W}, and need only to modify some of their treatments to match the Neumann boundary conditions such that integration by parts holds true. Their proof adopted the following intrinsic energy:
\[Q_1=\int_{\Om}d^2(u_1,u_2)dx\]
and
\[Q_2=\int_{\Om}|\P \n_2 u_2-\n_1 u_1|^2dx=\int_{\Om}|\Phi|^2dx,\]
and applies the geometric energy method to achieve a Gronwall-type inequality for $Q_1+Q_2$, from which the uniqueness of solutions follows. For the sake of simplicity and fluency, some notations about moving frame and parallel transportation as well as some critical lemmas are given in Appendix \ref{ref-B}, for more details we refer to \cite{S-W}. Now, we turn to presenting the proof.
	
\begin{proof}[\textbf{The proof of the uniqueness in Theorem \ref{thm1}}]\
		
Let $u_1, u_2: \Om\times[0,T]\to \U^2$ be two solution to \eqref{S-eq} in space $\S$. An important fact is that the uniqueness is a local property. Namely, once we know $u_1 = u_2$ on a small time interval $[0, T^\prime]$, then we can prove $u_1 = u_2$ on the whole interval $[0,T]$ by repeating the
argument. Therefore, we only need to prove the uniqueness in a small interval $[0,T^\prime]$. Our proof divides into three steps.
		
\medskip
\noindent\emph{Step 1: Estimate of the distance $d(u_1,u_2)$.}\

Since $u_l\in\S$, the embedding Lemma \ref{A-S} implies $u_l\in C^0([0,T], H^2(\Om))$, and hence $u_l\in C^0(\bar{\Om}\times[0,T])$. Therefore,
\[\norm{u_l-u_0}_{C^0(\bar{\Om})}(t)\leq C\norm{u_l-u_0}_{H^2(\Om)}(t)\to 0,\]
as $t\to 0$.
		
By using the fact $\U^2$ is of bounded geometry and taking $t\leq T^\prime$ with $T^\prime$ small enough, then there exists a constant $C$ such that
\[d(u_l,u_0)(t)\leq C\norm{u_l-u_0}_{C^0(\bar{\Om}\times[0,T^\prime])}<\frac{\pi}{4},\]
for $0<t\leq T^\prime$ and $l=1,2$. This step guarantees the parallel transportation between two solutions $u_1$ and $u_2$ can be well-defined (to see Sectin \ref{s-B}).
\medskip
		
\noindent\emph{Step 2: Estimate of $Q_1$.}\

Let $d:\U^2\times\U^2\to \Real$ be the distance function on $\U^2$, and $\tilde{\n}=\n^{\U^2}\otimes\n^{\U^2}$ be the product connection on $\U^2\times\U^2$. Suppose $\{x_1,x_2,x_3\}$ is the coordiniates of $\Om$, and $\n_i=\n_{\frac{\p}{\p x_i}}$ is the bull-back connection on $u^*_{1}T\U^2$ (or $u^*_{2}T\U^2$). Then a direct calculation shows
\begin{equation}
\begin{aligned}
\frac{\p}{\p t}\int_{\Om}d^2(u_1,u_2)dx=&\int_{\Om}\<\tilde{\n}d^2,( \n_i(J\n_iu_1), \n_i(J\n_i u_2))\>dx\\
=&\int_{\Om}\<\tilde{\n}d^2(\cdot, u_2),\n_i(J\n_iu_1)\>dx+\int_{\Om}\<\tilde{\n}d^2(u_1,\cdot),\n_i(J\n_iu_2)\>dx\\
=&\int_{\Om}\frac{\p}{\p x_i}\<\tilde{\n}d^2(\cdot, u_2),J\n_iu_1\>dx-\int_{\Om}\<\n_i\tilde{\n}d^2(\cdot, u_2),J\n_iu_1\>dx\\
&+\int_{\Om}\frac{\p}{\p x_i}\<\tilde{\n}d^2(u_1,\cdot),J\n_iu_1\>dx-\int_{\Om}\<\n_i\tilde{\n}d^2(u_1,\cdot),J\n_iu_1\>dx\\
=&-(\int_{\Om}\<\n_i\tilde{\n}d^2(\cdot, u_2),J\n_iu_1\>dx+\int_{\Om}\<\n_i\tilde{\n}d^2(u_1,\cdot),J\n_iu_1\>dx)\\
=&-\int_{\Om}\<\n\tilde{\n}d^2, (J\n u_1,J\n u_2)\>dx\\
=&-\int_{\Om}\tilde{\n}^2d^2(X,Y)dx,
\end{aligned}
\end{equation}
where $X=(\n u_1, \n u_2)$ and $Y=(J \n u_1, J \n u_2)$. Moreover, here we have used the Stokes formula:
\[\int_{\Om}\frac{\p}{\p x_i}\<\tilde{\n}d^2(\cdot, u_2),J\n_iu_1\>dx=\int_{\p\Om}\<\tilde{\n}d^2(\cdot, u_2),J(\n_iu_1\cdot\nu_i)\>ds=0,\]
and
\[\int_{\Om}\frac{\p}{\p x_i}\<\tilde{\n}d^2(u_1, \cdot),J\n_iu_2\>dx=\int_{\p\Om}\<\tilde{\n}d^2(u_1, \cdot),J(\n_i u_2\cdot\nu_i)\>ds=0,\]
since the Neumann boundary conditions: $\sum_{i=1}^3 \n_i u_l \cdot\nu_i=0$ for $l=1, 2$, where $\nu$ is the normal outer vector of $\p\Om$ and $ds$ is the area element of $\p \Om$.
		
Therefore, Lemma \ref{es-dist} gives
\[\frac{1}{2}\frac{\p}{\p t}Q_1\leq Q_2+C(\norm{u_1}^2_{H^3(\Om)}+\norm{u_2}^2_{H^3(\Om)})Q_1.\]
		
\medskip
		
\noindent\emph{Step 3: Estimate of $Q_2$.}

Let $\{e_\al\}_{\al=1}^2$ be a local frame of the pull-back bundle $u^*_1T\U^2$ over $\Om\times[0,T]$, such that the complex structure $J$ in this frame is reduced to $J_0=\sqrt{-1}$. Let $\n_l=u_l^*\n^{\U^2}$, the parallel transportation $\P$ and
\[\Phi=\P \n_2 u_2-\n_1 u_1=\phi_2-\phi_1\]
be defined in Appendix \ref{ref-B}. Then, a direct calculation shows
\begin{equation}
\begin{aligned}
&\frac{1}{2}\frac{\p}{\p t}\int_{\Om}|\P \n_2 u_2-\n_1 u_1|^2dx=\frac{1}{2}\frac{\p}{\p t}\int_{\Om}|\Phi|^2dx\\
=&\int_{\Om}\<\Phi,\n_{1,t}\Phi\>dx=\int_{\Om}\<\Phi,(\n_{1,t}-\n_{2,t})\phi_2\>dx\\
&-\int_{\Om}\<\Phi, (\n_1\n_{1,i}-\n_2\n_{2,i})J_0\phi_{2,i}\>dx+\int_{\Om}\<\Phi, \n_1(J_0\n_{1,i} \Phi_i)\>dx\\
=&I+II+III.
\end{aligned}
\end{equation}

Next, we estimate the above three terms step by steps.
\begin{align*}
|I|\leq&Q_{2}+C\int_{\Om}|(\n_{1,t}-\n_{2,t})\phi_2|^2dx\\
\leq &Q_{2}+C\norm{\n_2u_2}^2_{L^\infty(\Om)}\int_{\Om}d^2(u_1,u_2)(|\n_t u_2|^2+|\n_t u_1|^2)dx\\
\leq &Q_{2}+C\norm{u_2}^2_{H^3(\Om)}\norm{d(u_1, u_2)}^2_{L^4}(\norm{\frac{\p u_1}{\p t}}^2_{L^4}+\norm{\frac{\p u_2}{\p t}}^2_{L^4})\\
\leq &Q_{2}+C\norm{u_2}^2_{H^3(\Om)}(\norm{\frac{\p u_1}{\p t}}^2_{H^1}+\norm{\frac{\p u_2}{\p t}}^2_{H^1})\norm{d(u_1, u_2)}^2_{H^1}\\
\leq &C\norm{u_2}^2_{H^3(\Om)}(\norm{\frac{\p u_1}{\p t}}^2_{H^1}+\norm{\frac{\p u_2}{\p t}}^2_{H^1}+1)(Q_1+Q_2).
\end{align*}
Here, in the second line and the last line of the above inequality, we have used Lemma \ref{es-con} and the fact
\[|\n d(u_1,u_2)|\leq |\Phi|=|\P \n_2 u_2-\n_1u_1|\]
respectively.		
By taking a analogous argument to that for term $I$ and using Lemma \ref{es-con}, we have
\[|II|\leq C(\norm{u_2}^2_{H^3(\Om)}+\norm{u_2}^2_{H^3(\Om)}+1)\norm{u_2}^2_{H^3}(Q_1+Q_2).\]
		
For the term $III$, a simple calculation shows
\begin{align*}
III=&\int_{\Om}\mbox{div}\<\Phi, J_0\n_1\cdot\Phi\>dx-\int_{\Om}\<\n_1\cdot\Phi, J_0\n_1\cdot\Phi\>dx=0.
\end{align*}
Here we denote $\n_1\cdot \Phi=\n_{1,i}\Phi_i$ and have used the Stokes formula:
\begin{align*}
&\int_{\Om}\mbox{div}\<\Phi, J_0\n_1\cdot\Phi\>dx\\
=&\int_{\p \Om}\<\P \n_{2,i} u_2\cdot \nu_i-\n_1 u_{1,i}\cdot\nu_i, J_0\n_1\cdot\Phi\>ds=0
\end{align*}
since we have the Neumann boundary condition
\[\sum_{i=1}^3 \n_i u_l \cdot\nu_i=0,\quad l=1,2.\]
		
Therefore, by combining the estimates in Step 2 and Step 3, we have that for all $0<t\leq T^\prime$ there holds true
\[\frac{\p}{\p t}(Q_1+Q_2)\leq C(Q_1+Q_2),\]
which implies the uniqueness of solutions. Thus, the proof is completed.
\end{proof}

\medskip
\appendix
\renewcommand{\appendixname}{Appendix~\Alph{section}}
	
\section{Locally regular estimates of $u_\ep$}\label{ref-A}
In this section, we establish the regular estimates of the solution $v:\Om\times [0,T)\to \Real^3$ to the following uniform parabolic equation:
\begin{equation}\label{h-eq}
\begin{cases}
\frac{\p v}{\p t}-\ep\De v-u\times \De v=f(x, t), \quad (x,t)\in \Om\times [0,T],\\[1ex]
v(x,0)=v_0:\Om\to \Real^3,\quad \frac{\p v}{\p \nu}|_{\p\Om\times [0, T]}=0.
\end{cases}
\end{equation}
where
\begin{equation}\label{cond1}
u\in L^\infty ([0,T], H^3(\Om))\cap C^0(\Om\times[0,T])
\end{equation}
and
\begin{equation}\label{cond2}
f(x,t)\in L^{2}([0, T], H^1(\Om)).
\end{equation}
	
Our main result on locally regular estimates of solution $v$ to the above equation \eqref{h-eq} is as follows.
\begin{thm}\label{reg-es}
Let $v\in W^{2,1}_{2}(\Om\times [0,T])$ is a strong solution to \eqref{h-eq} satisfying condtions \eqref{cond1} and \eqref{cond2}. Then  $v\in L^\infty_{loc} ((0,T],H^2(\Om))\cap L^2_{loc}((0,T],H^3(\Om))$. Moreover, for any $\de>0$, there exists a positive constant $C(\de)$ depending on $\norm{u}_{L^\infty([0,T], H^3(\Om))}$ such that there holds
\begin{align*}
%&\norm{v}_{L^\infty([\de, T], H^1(\Om))}+%
&\norm{v}_{L^2([\de,T],H^2(\Om))}+\norm{\frac{\p v}{\p t}}_{L^{2}([\de,T],L^2(\Om))}
\leq C(\de)\(\norm{v}_{L^2([0,T]\times \Om)}+\norm{f}_{L^2(\Om\times[0,T])}\)
\end{align*}
and
\begin{align*}
%&\norm{v}_{L^\infty([\de, T], H^2(\Om))}+%
&\norm{v}_{L^2([\de,T],H^3(\Om))}+\norm{\frac{\p v}{\p t}}_{L^{2}([\de,T],H^1(\Om))}
\leq C(\de)\(\norm{v}_{L^2([0,T], H^1(\Om))}+\norm{f}_{L^2([0,T], H^1(\Om))}\).
\end{align*}
\end{thm}
\begin{proof}
Let $\eta(t)$ be a smooth cut-off function such that $\mbox{supp}\eta\subset (0,T]$ and $\eta\equiv1$ on $[\de,T]$ for any $\de>0$. Then $\eta v$ is a strong solution of the following equation
\begin{equation}\label{H-eq}
\begin{cases}
\frac{\p \om}{\p t}-\ep\De \om -u\times \De \om=\tilde{f},\quad&\text{(x,t)}\in\Om\times[0,T],\\[1ex]
\frac{\p\om}{\p \nu}|_{\p \Om\times[0,T]}=0,\quad \om(x,0)=0,\quad &\text{x}\in\Om.\\[1ex]
\end{cases}
\end{equation}
Here
\[\tilde{f}=\eta f-\frac{\p \eta}{\p t}v,\]
it is easy to see that the assumptions in Theorem \ref{reg-es} implies
\[\tilde{f}\in L^2([0,T], H^{1}(\Om)).\]

Next, we claim there exists a solution $w\in L^\infty([0,T],H^2(\Om))\cap L^2([0,T],H^3(\Om))$ to the above equation \eqref{H-eq}, by using Galerkin approximation methold. Namely we consider the Galerkin approximation equation to \eqref{H-eq}
\begin{equation}\label{G-H-eq}
\begin{cases}
\frac{\p\om^n}{\p t}-\ep\De \om^n -P_n(u\times \De \om^n)=P_n(\tilde{f}),\quad &\text{(x,t)}\in\Om\times[0,T],\\[1ex]
\om^n(x,0)=0,\quad &\text{x}\in\Om,\\[1ex]
\end{cases}
\end{equation}
where the Galerkin projection $P_n$ is defined in Section \ref{s: pre}. Under the assumptions of $u$ and $\tilde{f}$, one can show there exists a unique solution $w^n(x,t)=\sum_{i=1}^{n}g_i(t)f_i(x)$ to \eqref{G-H-eq} on $\Om\times[0,T]$ (cf. \cite{YaZ}).

Then by taking $w^n$, $\De w^n$ and $\De^2 w^n$ as test functions to \eqref{G-H-eq}, a simple calculation shows
\begin{align*}
&\frac{\p}{\p t}\int_{\Om}|\om^n|^2dx+\ep\int_{\Om}|\n\om^n|^2dx\leq C(\ep) \norm{u}_{L^\infty([0,T], H^3)}\int_{\Om}|\om^n|^2dx+C(\ep)\int_{\Om}|\tilde{f}|^2dx,\\
&\frac{\p}{\p t}\int_{\Om}|\n \om^n|^2dx+\ep\int_{\Om}|\De\om^n|^2dx\leq C(\ep)\int_{\Om}|\tilde{f}|^2dx,\\
&\frac{\p}{\p t}\int_{\Om}|\De \om^n|^2dx+\ep\int_{\Om}|\n\De\om^n|^2dx\leq C(\ep) \norm(u)_{L^\infty([0,T], H^3)}\int_{\Om}|\De \om^n|^2dx+C(\ep)\int_{\Om}|\n \tilde{f}|^2dx.
\end{align*}
The Gronwall inequality gives
\begin{align}
&\sup_{0\leq t\leq T}\norm{w^n}^2_{H^1}+\ep\int_{0}^{T}\int_{\Om}|\De \om^n|^2dxdt\leq C(\ep, \norm{u}_{L^\infty([0,T],H^3)}, T)\int_{0}^{T}\int_{\Om}|\tilde{f}|^2dxdt,\\
&\sup_{0\leq t\leq T}\int_{\Om}|\De \om^n|^2dx+\ep\int_{0}^{T}\int_{\Om}|\n\De\om^n|^2dxdt\leq C(\ep, \norm{u}_{L^\infty([0,T],H^3)}. T)\int_{0}^{T}\int_{\Om}|\n\tilde{f}|^2dxdt.
\end{align}
By using Equation \eqref{G-H-eq} again, one obtains
\begin{align*}
&\sup_{0\leq t\leq T}\norm{w^n}^2_{H^1}+\ep\(\int_{0}^{T}\int_{\Om}|\frac{\p w^n}{\p t}|^2dxdt+\int_{0}^{T}\norm{w^n}^2_{H^2}dt\)\\
\leq &C(\ep, \norm{u}_{L^\infty([0,T],H^3)}, T)\int_{0}^{T}\int_{\Om}|\tilde{f}|^2dxdt,\\	
&\sup_{0\leq t\leq T}\int_{\Om}|\frac{\p w^n}{\p t}|^2dx+\ep\int_{0}^{T}\int_{\Om}|\n\frac{\p w^n}{\p t}|^2dxdt\\
\leq &C(\ep, \norm{u}_{L^\infty([0,T],H^3)}, T)\int_{0}^{T}\int_{\Om}|\tilde{f}|^2+|\n\tilde{f}|^2dxdt.
\end{align*}
Therefore by letting $n\to \infty$, $w^n$ converges to a solution $w$ of \eqref{H-eq} as we claimed, which satisfies
\begin{align}
&\norm{\om}_{L^\infty([0, T], H^1(\Om))}+\norm{\om}_{L^2([0,T],H^2(\Om))}+\norm{\frac{\p \om}{\p t}}_{L^{2}([0,T],L^2(\Om))}\leq C(\ep)\norm{\tilde{f}}_{L^2([0,T]\times\Om)},\\
&\norm{\om}_{L^\infty([0, T], H^2(\Om))}+\norm{\om}_{L^2([0,T],H^3(\Om))}+\norm{\frac{\p \om}{\p t}}_{L^{2}([0,T],H^1(\Om))}
\leq C(\ep)\norm{\tilde{f}}_{L^2([0,T], H^1(\Om))}.
\end{align}

Then the uniqueness of strong solutions to \eqref{H-eq} implies $\eta v=\om$. Therefore the desired result is proved.
\end{proof}

Now, we apply the above Theorem \ref{reg-es} to show the local estimates of solution $\frac{\p u_\ep}{\p t}$ to \eqref{w-eq1} which can be summarized as the following theorem.
	
\begin{thm}\label{A.3}
The solution $u_\ep$ to \eqref{PS-eq} obtained in Theorem \ref{uin-ex-PS} satisfies
\[\frac{\p u_\ep}{\p t}\in L^2_{loc}((0,T_\ep), H^3(\Om))\]
and
\[\frac{\p^2 u_\ep}{\p t^2}\in L^2_{loc}((0,T_\ep), H^1(\Om)).\]
\end{thm}
	
\begin{proof}
Since $\frac{\p u_\ep}{\p t}$  satisfies the following equation:
\begin{equation}\label{w-eq3}
\begin{cases}
\frac{\p v}{\p t}-\ep\De v-u_\ep\times\De v=f(u_\ep, \n v, v),\\[1ex]
v(x,0)=\frac{\p u_\ep}{\p t}|_{t=0},\quad \frac{\p v}{\p t}|_{\p\Om\times (0, T_\ep)}=0.
\end{cases}
\end{equation}
Here
\[f(u_\ep, \n v, v)=v\times \De u_\ep+2\ep\<\n u_\ep,\n v\>u_\ep+\ep |\n u_\ep|^2v.\]
By the above Theorem \ref{reg-es}, we only need to check that $u_\ep$ and $f(u_\ep, \n v,v)$ satisfies the conditions \eqref{cond1} and \eqref{cond2} respectively.
		
Since $u_\ep\in L^\infty([0,T], H^3(\Om))$ and $\frac{\p u_\ep}{\p t}\in L^2([0,T], L^2(\Om))$ for all $T<T_\ep$, then $(2)$ of Lemma \ref{A-S} implies
\[u_{\ep}\in C^0([0,T], H^2(\Om)).\]
Indeed, for any $(x,t)$ and $(x_0,t_0)$ we have
\begin{equation}
\begin{aligned}
|u_\ep(x,t)-u_\ep(x_0,t_0)|\leq&|u_\ep(x,t)-u_\ep(x_0,t)|+|u_\ep(x_0,t)-u_\ep(x_0,t_0)|\\
\leq &\sup_{\Om}|\n u_\ep|(\cdot,t)|x-x_0|+|u_\ep(x_0,t)-u_\ep(x_0,t_0)|\\
\leq &C\norm{u_\ep}_{L^\infty([0,T], H^3(\Om))}|x-x_0|+C\norm{u_\ep(\cdot, t)-u_\ep(\cdot, t_0)}_{H^2(\Om)}.
\end{aligned}
\end{equation}
This implies $u_\ep\in C^0(\Om\times [0,T])$.
		
Next, we want to show $f(u_\ep,\n v, v)\in L^2([0,T], H^1(\Om))$. A simple calculations shows
\[\int_0^T\int_{\Om}|f|^2dxdt\leq C(1+\ep)(\norm{u_\ep}^4_{L^\infty([0,T], H^3)}+1)\norm{\frac{\p u_\ep}{\p t}}^2_{L^2([0,T],H^1)}\leq C(T),\]
and
\begin{align*}
\n f=&\n \frac{\p u_\ep}{\p t}\times \De u_\ep+\frac{\p u_\ep}{\p t}\times \n \De u_\ep\\
&+\n^2\frac{\p u_\ep}{\p t}\#\n u_\ep\# u_\ep+\n \frac{\p u_\ep}{\p t}\# \n^2 u_\ep\# u_\ep\\
&+\n \frac{\p u_\ep}{\p t}\# \n u_\ep\# \n u_\ep+\frac{ \p u_\ep}{\p t}\# \n^2 u_\ep\#\n u_\ep,\\
\end{align*}
where ``$\#$" denotes the linear contraction. Thus, we have
\[\int_0^T\int_{\Om}|\n f|^2dxdt\leq C(\norm{u_\ep}^4_{L^\infty([0,T], H^3)}+\norm{u_\ep}^2_{L^2([0,T], H^3)}+1)\norm{\frac{\p u_\ep}{\p t}}_{L^2([0,T], H^2)}\leq C(T).\]
Here, we have used the estimate \eqref{es-PS}. Therefore, from Theorem \ref{reg-es} we can obtain the desired results.
\end{proof}
\medskip
	
\section{The Schr\"{o}dinger flow in moving frame and parallel transportation}\label{ref-B}
\subsection{The Schr\"{o}dinger flow in moving frame}
Let $\Om$ be a smooth bounded domain in $\Real^3$. Suppose that $u: \Om\times [0,T]\to \U^2$ is a solution to the Schr\"{o}dinger flow
\begin{equation}\label{S-eq1}
\begin{cases}
\p_tu =J(u)\tau(u),\quad\quad&\text{(x,t)}\in\Om\times \Real^+,\\[1ex]
\frac{\p u}{\p \nu}=0, &\text{(x,t)}\in\p\Om\times \Real^+,\\[1ex]
u(x,0)=u_0: \Om\to \U^2,
\end{cases}
\end{equation}
where $J(u)=u\times$. We are going to rewrite the above equation in a chosen gauge of the pull-back bundle $u^*T\U^2$ over $\Om\times[0,T]$.
	
Let $\n^{\U^2}$ be the connection on $\U^2$ and $\n=u^*\n^{\U^2}$ be the pull-back connection on $u^*T\U^2$. Let $\{x_1,x_2,x_3, t\}$ be the canonical coordinates on $\Om\times[0,T]$, denote $\n_{t}=\n_{\frac{\p }{\p t}}$ and $\n_{i}=\n_{\frac{\p }{\p x_1}}$ for $i=1,2,3$. Recall that the tension field $\tau(u)=\tr\n^2 u=\n_i\n_i u$, we can write the equation \eqref{S-eq1} in the form
\[\n_t u=J(u)\n_i\n_i u=\n_i(J(u)\n_i u).\]
	
Let $\{e_\al\}_{\al=1}^2$ be a local frame of the pull-back bundle $u^*T\U^2$ over $\Om\times[0,T]$, such that the complex structure $J$ in this frame is reduced to $J_0=\sqrt{-1}$.
	
Let $\n_i u=u_i^\al e_\al$, where $i=1,2,3$. The Neumann condition on boundary in \eqref{S-eq1} is equivalent to
\[\sum_{i=1}^3u_i^\al\cdot \nu_i|_{\p\Om\times [0,T]}=0\]
for $\al=1,2$, where $\nu=(\nu_1,\nu_2,\nu_3)$ be the outer normal vector of $\p \Om$.

\subsection{Parallel transportation and some lemmas}\label{s-B}
Let $B\subset \U^2$ be an open geodesic ball with radius $<\frac{\pi}{2}$. Then for any $y_1,\, y_2\in B$, there exists a unique minimizing geodesic $\ga(s):[0,1]\to \U^2$ connecting $y_1$ and $y_2$, and let $\P:T_{y_2}\U^2\to T_{y_1}\U^2$ be the linear map given by parallel transport along $\ga$. Let $\{e_1(s),e_2(s)\}$ be the frame gotten by parallel transport along $\ga$, set $e_\al(y_1)=e_\al(0)$ and $e_\al(y_2)=e_\al(1)$. Then, for any $X=X^\al(y_2)e_\al(1)\in T_{y_2}\U^2$, the above linear map $\P$ has the following formula
	\[\P X=X^\al(y_2)e_\al(1).\]
	
Let $d:\U^2\times\U^2\to \Real$ be the distance function on $\U^2$, and $\tilde{\n}=\n^{\U^2}\otimes\n^{\U^2}$ be the product connection on $\U^2\times\U^2$. We have the following estimates for gradient and Hessian of the distance function, whose proof can be found in \cite{CJG,S-W}.
\begin{lem}\label{es-dist}
Suppose $X=(X_1,X_2)$ and $Y=(Y_1,Y_2)$ be two vectors in $T_{y_1}\U^2\times T_{y_2}\U^2$such that $d(y_1,y_2)<\frac{\pi}{2}$, then
\begin{itemize}
\item[$(1)$] $\frac{1}{2}\tilde{\n}d^2(X)=\<\ga^\prime(0),\P X_2-X_1\>$,\\
			
\item[$(2)$] $\frac{1}{2}|\tilde{\n}^2d^2(X,Y)|\leq |\P X_2-X_1||\P Y_2-Y_1|+Cd^2(y_1,y_2)(|X_1|+|X_2|)(|Y_1|+|Y_2|)$.
\end{itemize}
\end{lem}
	
\medskip
On the other hand, let $u_l:\Om\times [0,T]\to \U^2$, $l=1,2$, with
$$\sup_{\Om\times[0,T]} d(u_1(x,t),u_2(x,t))<\frac{\pi}{2}$$
and denote $\n_l=u_l^*\n^{\U^2}$. Then, for any $(x,t)\in \Om\times[0,T]$ there exists a unique minimizing geodesic $\ga_{(x,t)}(s):[0,1]\to \U^2$ connecting $u_1(x,t)$ and $u_2(x,t)$. More precisely, we define a map $U:\Om\times [0,T]\times[0,1]\to \U^2$ such that $U(x,t,s)=\ga_{(x,t)}(s)$, then $u_l^*T\U^2=U^*T\U^2|_{s=l-1}$ and $\n_l=U^*\n^{\U^2}|_{s=l-1}$. Therefore, we can define a global bundle isomorphism $\P : u^*_2T\U^2 \to u^*_1T\U^2 $ by the parallel transportation along each geodesic. And hence, $\P$ can be extended naturally to a bundle isomorphism from $u^*_2T\U^2\otimes T^*\Om$ to $u^*_1T\U^2\otimes T^*\Om$.
	
Let $\{e_1,e_2\}$ be a fixed local frame of bundle $u_1^*T\U^2$ such that $J(u_1)=\sqrt{-1}$. For each point $(x,t)$, we parallel transport this frame to get a moving frame $\{e_1(s),e_2(s)\}$ along the geodesic $\ga_{(x,t)}(s)$, and set $e_{1,\al}=e_\al(0)$ and $e_{2,\al}=e_\al(1)$ for $\al=1,2$.  Under this local frame $\{e_1(1), e_2(1)\}$ of $u^*_2 T\U^2$, we still have
\[J(u_2)=\sqrt{-1},\]
since $\n_{\frac{\p \ga}{\p s}} J(\ga)=\frac{\p }{\p s} J\circ \ga=0$ and $J\circ \ga(0, x,t)=\sqrt{-1}$.
	
On the other hand, if we denote $\n_l u_l=u^\al_{l,i}e_{l,\al}\otimes dx^i$ and set $\phi_l=u^\al_{l,i}e_{1,\al}\otimes dx^i$, then
\[\P\n_2 u_2=\P u^\al_{2,i} e_{2,\al}\otimes dx^i=u^\al_{2,i} e_{1,\al}\otimes dx^i=\phi_2,\]
and hence
\[\Phi:=\P \n_2 u_2-\n_1 u_1=(u^\al_{2,i}-u^\al_{1,i})e_{1,\al}\otimes dx^i=\phi_2-\phi_1.\]
\medskip
	
Denote the difference of the two connections by
\[B=\n_2-\n_1=(\<\n_2 e_\al(1), e_{\beta}(1)\>-\<\n_2 e_\al(0), e_{\beta}(0)\>)e_\beta(0),\]
which is a tensor. The following estimates for the difference of connections is essential to control the energy $\int_{\Om}|\Phi|^2dx$ in the proof of the uniqueness, whose proof can be found in \cite{S-W}.
\begin{lem}\label{es-con}
The exists constant C independing on $u_1$ and $u_2$, such that the following bounds hold true.
\begin{itemize}
	\item[$(1)$] $|B_t|=|\n_{2,t}-\n_{1,t}|\leq C(|\n_t u_1|+|\n_t u_2|)d(u_1,u_2)$,
	\item[$(2)$] $|B_i|=|\n_{2,i}-\n_{1,i}|\leq C(|\n_i u_1|+|\n_i u_2|)d(u_1,u_2)$,
	\item[$(3)$] $|\n_{2,k}\n_{2,i}-\n_{1,k}\n_{1,i}|\leq C(|\Phi|+(|\n_1^2 u_1|+|\n_2^2 u_2|+1)d(u_1,u_2))$.
\end{itemize}
Here $i,k=1,2,3$.
\end{lem}

\medskip
\noindent {\it\bf{Acknowledgements}}: The authors are supported partially by NSFC (Grant No.11731001). The author Y. Wang is supported partially by NSFC (Grant No.11971400), and Guangdong Basic and Applied Basic Research Foundation (Grant No. 2020A1515011019).

\end{document}